\documentclass[a4paper,11pt]{article}
\usepackage{amsmath,amssymb,enumerate,color}
\usepackage{amsthm,color}
\usepackage{calc}
\usepackage{txfonts}

\usepackage{hyperref}
\hypersetup{
setpagesize=false,
 bookmarksnumbered=true,%
 bookmarksopen=true,%
 colorlinks=true,%
 linkcolor=blue,
 citecolor=red
}

\newcommand{\R}{\mathbb{R}}

\newcommand{\N}{\mathbb{N}}
\newcommand{\ep}{\varepsilon}
\newcommand{\pa}{\partial}

\newcommand{\lr}[1]{{}\langle{}#1{}\rangle{}}

\newtheorem{theorem}{Theorem}[section]
\newtheorem{lemma}[theorem]{Lemma}
\newtheorem{proposition}[theorem]{Proposition}
\newtheorem{corollary}[theorem]{Corollary}
\theoremstyle{remark}
\newtheorem{remark}{Remark}[section]
\theoremstyle{definition}

\newtheorem{definition}[theorem]{Definition}

\numberwithin{equation}{section}

\begin{document}
\begin{center}
\Large{{\bf
Supersolutions for parabolic equations with unbounded diffusion 
and 
\\
its applications to 
some classes of parabolic and hyperbolic equations
}}
\end{center}

\vspace{5pt}

\begin{center}
Motohiro Sobajima%
\footnote{
Department of Mathematics, 
Faculty of Science and Technology, Tokyo University of Science,  
2641 Yamazaki, Noda-shi, Chiba, 278-8510, Japan,  
E-mail:\ {\tt msobajima1984@gmail.com}}
and 
Yuta Wakasugi%
\footnote{
Graduate School of Science and Engineering, Ehime University, 
3, Bunkyo-cho, Matsuyama, Ehime, 790-8577, Japan, 
E-mail:\ {\tt wakasugi.yuta.vi@ehime-u.ac.jp}.}
\end{center}

\newenvironment{summary}{\vspace{.5\baselineskip}\begin{list}{}{%
     \setlength{\baselineskip}{0.85\baselineskip}
     \setlength{\topsep}{0pt}
     \setlength{\leftmargin}{12mm}
     \setlength{\rightmargin}{12mm}
     \setlength{\listparindent}{0mm}
     \setlength{\itemindent}{\listparindent}
     \setlength{\parsep}{0pt}
     \item\relax}}{\end{list}\vspace{.5\baselineskip}}
\begin{summary}
{\footnotesize {\bf Abstract.}
This paper is concerned with 
supersolutions to parabolic equations 
of the form 
\begin{equation}\label{abst:eq1}
\pa_t U (x,t)-D(x)\Delta U(x,t)=0, 
\quad (x,t)\in \R^N\times (0,\infty),	
\end{equation}
where $D\in C(\R^N)$ is positive. 
Under the behavior of 
the diffusion coefficient $D$ 
with polynomial order at spatial infinity, 
a family of supersolutions to \eqref{abst:eq1} 
with slowly decaying property at spatial infinity 
is provided. As a first application, 
weighted $L^2$ type decay estimates for 
the initial-boundary value problem of 
the parabolic equation
\begin{equation}\label{abst:eq2}
\begin{cases}
\pa_tv(x,t)-D(x)\Delta v(x,t)=0, 
& (x,t)\in \Omega\times (0,\infty),	
\\
v (x,t)=0, 
& (x,t)\in \pa\Omega\times (0,\infty),	
\\
v(x,0)=v_0(x), 
& x\in \Omega
\end{cases}
\end{equation}
are proved. The second application is 
the study of the exterior problem of 
wave equations with space-dependent damping terms of type 
\begin{equation}\label{abst:eq3}
\begin{cases}
\pa_t^2 u (x,t)-\Delta u(x,t)+a(x)\pa_tu(x)=0, 
& (x,t)\in \Omega\times (0,\infty),	
\\
u (x,t)=0, 
& (x,t)\in \pa\Omega\times (0,\infty),	
\\
u(x,0)=u_0(x), 
\ 
\pa_tu(x,0)=u_1(x), 
& x\in \Omega.
\end{cases}
\end{equation}
By using supersolution provided above with $a(x)=D(x)^{-1}$, 
energy estimates for \eqref{abst:eq3} with polynomial weight
and diffusion phenomena are shown. There is a slight improvement comparing to the one in \cite{SoWa19_CCM} about the assumption of the initial data. 
}
\end{summary}


\section{Introduction}
In this paper we consider 
positive supersolutions of the following parabolic equation 
\begin{align}\label{eq:D-heat}
\begin{cases}
\pa_t U (x,t)-D(x)\Delta U(x,t)=0, 
&(x,t)\in \R^N\times (0,\infty),	
\\
U (x,0)=w(x),
&x\in \R^N,	
\end{cases}
\end{align}
where $N\in \N$ and the diffusion coefficient $D$ 
satisfies
\begin{align}\label{eq:D-ass}
D\in C(\R^N), \quad 
D(x)>0,\quad 
\lim_{|x|\to \infty}\Big(|x|^{-\alpha}D(x)\Big)=D_0 
\end{align} 
for some constants $\alpha\in \R$ and $D_0>0$. That is, 
we assume 
$D_1\lr{x}^{\alpha}\leq D(x)\leq D_2\lr{x}^{\alpha}$ 
for some $D_1, D_2>0$, where $\lr{x}=\sqrt{1+|x|^2}$.
The initial value $w(x)$ also behaves polynomially at spatial infinity.
We emphasize that the diffusion coefficient $D$ 
is allowed to be unbounded ($\alpha>0$) and also decreasing ($\alpha<0$).

The notion of subsolutions and supersolutions for 
elliptic and parabolic problems are well-known. 
They are essentially provided from the maximum principle 
of the corresponding problems as typified by the positivity of the solutions.
Nowadays, so-called supersolution-subsolution methods for 
such problems are understood as powerful tools 
to analyse the existence and uniqueness of solutions 
and to discover profiles of them.
(For the detail, 
see e.g., Gilbarg--Trudinger \cite{GTbook} for elliptic problems
and 
Quittner--Souplet \cite{SQbook} for parabolic problems, and their references therein). 

Here we shall focus our attention 
to the notion of supersolutions for parabolic problems. 
For instance, in the case of the following problem
\begin{align}\label{nheat}
\begin{cases}
\pa_tu(x,t)-\Delta u(x,t) =\big(u(x,t)\big)^p
&(x,t)\in \Omega\times (0,T),
\\
u(x,t)=0, 
& (x,t)\in \pa\Omega\times (0,T),
\\
u(x,0)=u_0(x)\geq 0,
& x\in \Omega,
\end{cases}
\end{align}
the function $\overline{u}$ $\in C^2(\Omega\times[0,T))$
is called supersolution of \eqref{nheat} if 
\[
\begin{cases}
\pa_t \overline{u}(x,t) -\Delta \overline{u}(x,t) \geq \big(\overline{u}(x,t)\big)^p
&(x,t)\in \Omega\times (0,T),
\\
\overline{u}(x,t)\geq 0, 
& (x,t)\in \pa\Omega\times (0,T),
\\
\overline{u}(x,0)\geq u_0(x),
& x\in \Omega.
\end{cases}
\]
Once we find a supersolution of \eqref{nheat}, 
then we can immediately obtain an estimate for solutions in a pointwise sense. 
By using this structure, Weissler \cite{Weissler84} proved 
single-point blowup of solutions to \eqref{nheat}.
The structure of supersolutions also can be found in 
the study of nonlinear parabolic systems.
We only quote Levine \cite{Levine91}, 
Lu-Sleeman \cite{LS94} and Ishige--Kawakami--Sier\.{z}\c{e}ga \cite{IKS16}. 

The references stated above suggest that 
a criterion of construction of supersolutions (subsolutions) 
for general problem 
enables us to reach a further detailed analysis 
of profiles of solutions. 

The purpose of the present paper is to 
give a family of supersolutions to the problem \eqref{eq:D-heat}
and to discuss applications of
them to weighted $L^2$ type decay estimates for 
initial-boundary value problems of parabolic equations 
and diffusion phenomena for the 
hyperbolic equations with space-dependent damping. 

To state our main result, 
we would give the definition of 
supersolutions to \eqref{eq:D-heat} as follows. 
\begin{definition}
For given initial value $w\in C(\Omega)$, a function $\overline{U}$ is said to be 
a supersolution of \eqref{eq:D-heat} if $\overline{U}\in C^2(\overline{\Omega}\times [0,\infty))$ satisfies 
\[
\begin{cases}
\pa_t \overline{U}(x,t) -D(x)\Delta \overline{U}(x,t) \geq 0, 
&(x,t)\in \Omega\times (0,\infty),
\\
\overline{U}(x,t)\geq 0, 
& (x,t)\in \pa\Omega\times (0,\infty),
\\
\overline{U}(x,0)\geq w(x),
& x\in \Omega.
\end{cases}
\]
\end{definition}

We are interested in supersolutions with 
polynomially decaying property at spatial infinity. 
Therefore, we do not handle supersolutions similar to Gaussian function $t^{-N/2}
\exp(-|x|^2/4t)$. 
The following assertion is the main result of this paper, 
which deals with the supersolutions of diffusion equation \eqref{eq:D-heat} 
with polynomially decaying property at spatial infinity. 
\begin{theorem}\label{thm:main}
Assume that $D$ satisfies \eqref{eq:D-ass} for $\alpha\in (-\infty,\min\{2,N\})$.
Then 
for $\sigma\in (0,\frac{N-\alpha}{2})$, there exist a supersolution
$\overline{U}_{D,\sigma}\in C^2(\overline{\Omega}\times[0,\infty))$ 
with $w(x)=\lr{x}^{-2\sigma}$
and positive constants $c_{D,\sigma}$, $C_{D,\sigma}$ 
and $C'_{D,\sigma}$
such that 
\begin{gather}
\label{eq:two-sided}
c_{D,\sigma}
\left(1+t+\lr{x}^{2-\alpha}\right)^{-\frac{2\sigma}{2-\alpha}}
\leq 
\overline{U}_{D,\sigma}(x,t)
\leq 
C_{D,\sigma}
\left(1+t+\lr{x}^{2-\alpha}\right)^{-\frac{2\sigma}{2-\alpha}},
\\
\label{eq:deri-bdd}
|\pa_t\overline{U}_{D,\sigma}(x,t)|
\leq 
C'_{D,\sigma}
\left(1+t+\lr{x}^{2-\alpha}\right)^{-\frac{2\sigma}{2-\alpha}-1}.
\end{gather} 
\end{theorem}
\begin{remark}
Here we point out about the case $\sigma=\frac{N-\alpha}{2}$. 
Formally, we consider the equation $\pa_tu=|x|^{\alpha}\Delta u$. 
This has the (super)solution
\[
(1+t)^{-\frac{N-\alpha}{2-\alpha}}\exp\left(-\frac{|x|^{2-\alpha}}{(2-\alpha)^2(1+t)}\right).
\]
In contrast, if $\sigma < \frac{N-\alpha}{2}$,
then the corresponding self-similar solution 
of $\pa_tu=|x|^{\alpha}\Delta u$ has the following form
\[
(1+t)^{-\frac{2\sigma}{2-\alpha}}
\exp\left(-\frac{|x|^{-\alpha}}{(2-\alpha)^2(1+t)}\right)
M\left(\frac{N-\alpha-2\sigma}{2-\alpha},\frac{N-\alpha}{2-\alpha};\frac{|x|^{-\alpha}}{(2-\alpha)^2(1+t)}\right)
\approx
\left(1+t+\frac{|x|^{2-\alpha}}{(2-\alpha)^2}\right)^{-\frac{2\sigma}{2-\alpha}},
\]
where $M(\cdot,\cdot;\cdot)$ is the Kummer confluent hypergeometric function 
(see Definition \ref{def:M}, below, and also \cite{SoWa19_CCM}). 
In view of the explicit representation of self-similar solution, 
the condition $\sigma \leq \frac{N-\alpha}{2}$ 
is required to ensure the positivity of self-similar solutions
and 
the restriction $\sigma < \frac{N-\alpha}{2}$ 
is necessary for the polynomially decaying profile of self-similar solutions.
\end{remark}

As a first application to supersolutions in Theorem \ref{thm:main}, 
we provide a weighted $L^2$-type decay estimate 
for the initial-boundary value problem of \eqref{eq:D-heat}, that is, 
\begin{align}\label{eq:D-heat-inibdry}
\begin{cases}
\pa_t v (x,t)-D(x)\Delta v(x,t)=0, 
&(x,t)\in \Omega\times (0,\infty),	
\\
v(x,t)=0, 
&(x,t)\in \pa \Omega\times (0,\infty),	
\\
v (x,0)=v_0(x), &x\in \Omega
\end{cases}
\end{align}
under the assumption \eqref{eq:D-ass}.
Here we consider the problem \eqref{eq:D-heat-inibdry} in an exterior domain 
in $\R^N$. 
In this case, the Friedrichs extension of 
the corresponding elliptic operator $D\Delta$ generates 
an analytic semigroup $\{T(t)\}_{t\geq 0}$ on a weighted $L^2$ space (see Section 4). 
We say that $v=T(t)v_0$ is the solution of \eqref{eq:D-heat-inibdry}.  
The properties of solutions to this problem, for  instance $L^p$-$L^q$ type estimates, can be found in literature 
(see e.g., 
Ioku--Metafune--Sobajima--Spina \cite{IMSS16}, 
Sobajima--Wakasugi \cite{SoWa16_JDE, SoWa18_ADE}). 
In the present paper, we will give 
the following assertion. 
\begin{proposition}\label{prop:heat-decay}
Assume that $D$ satisfies \eqref{eq:D-ass} with $\alpha\in (-\infty,\min\{2,N\})$	. 
Let $v_0$ satisfy $\lr{x}^{\sigma-\frac{\alpha}{2}}v_0\in L^2(\Omega)$
for some $\sigma\in (0,\frac{N-\alpha}{2})$. 
Then the solution $v$ of the problem \eqref{eq:D-heat-inibdry} satisfies 
\[
\big\|\lr{x}^{-\frac{\alpha}{2}}v(t)\big\|_{L^2(\Omega)}
\leq 
C(1+t)^{-\frac{\sigma}{2-\alpha}}\big\|\lr{x}^{\sigma-\frac{\alpha}{2}}v_0\big\|_{L^2(\Omega)}, \quad t\geq 0
\]
for some positive constant $C$ independent of $v_0$. 
\end{proposition}

\begin{remark}
Proposition \ref{prop:heat-decay} comes from the following formal computation 
via integration by parts twice: for positive $\Phi\in C^2(\R^{N+1})$, 
\begin{align*}
\frac{d}{dt}
\int_\Omega \frac{v^2}{\Phi D}\,dx
&=
2\int_\Omega \frac{v\pa_t v}{\Phi D}\,dx
-\int_\Omega \frac{v^2\pa_t\Phi}{\Phi^2 D}\,dx
\\
&=
2\int_\Omega \frac{v\Delta v}{\Phi}\,dx
-\int_\Omega \frac{v^2\pa_t\Phi}{\Phi^2 D}\,dx
\\
&=
-2\int_\Omega \left|\nabla \left(\frac{v}{\Phi}\right)\right|^2\Phi\,dx
-\int_\Omega \frac{v^2(\pa_t\Phi-D\Delta \Phi)}{\Phi^2 D}\,dx.
\end{align*}
From the above estimate, one can find 
a weighted $L^2$-estimate of $v$ through 
the $L^\infty$-estimate of supersolution $\Phi$. 
However, this is not clear because of the regularity of $v$ 
to verify the computation with integration by parts. 
Instead of the difficulty stated above, 
the desired estimate is proved via semigroup approach as an application of 
the Trotter-Kato approximation theorem. 
\end{remark}

The second application is the analysis of 
wave equations with 
space-dependent damping term:
\begin{equation}\label{eq:dw}
\begin{cases}
\pa_t^2 u (x,t)-\Delta u(x,t)+a(x)\pa_tu(x)=0, 
& (x,t)\in \Omega\times (0,\infty),	
\\
u (x,t)=0, 
& (x,t)\in \pa\Omega\times (0,\infty),	
\\
u(x,0)=u_0(x), 
\ 
\pa_tu(x,0)=u_1(x), 
& x\in \Omega, 
\end{cases}
\end{equation}
where $a(x)=D(x)^{-1}$, that is, $a\in C(\overline{\Omega})$ satisfying 
\begin{equation}\label{ass:a_intro}
a(x)>0, \quad \lim_{|x|\to \infty}\Big(|x|^{\alpha}a(x)\Big)=a_0\left(=D_0^{-1}\right). 
\end{equation}
If $a\equiv 1$, then \eqref{eq:dw} is the usual 
damped wave equation. This is motivated by the derivation 
by Cattaneo \cite{Ca58} and Vernotte \cite{Ve58} as the approximation of heat equation with finite propagation property. 
Actually, it is known that 
if $u_0$ and $u_1$ are compactly supported smooth function, 
then 
so-called diffusion phenomena occurs, 
that is, the solution of \eqref{eq:dw} 
for $t\gg 1$ can be approximated by the solution of 
\begin{equation}\label{eq:dw-heat}
\begin{cases}
a(x)\pa_tv(x)-\Delta v(x,t)=0, 
& (x,t)\in \Omega\times (0,\infty),	
\\
v (x,t)=0, 
& (x,t)\in \pa\Omega\times (0,\infty),	
\\
v(x,0)=u_0(x)+a(x)^{-1}u_1(x),
& x\in \Omega 
\end{cases}
\end{equation}
(which is equivalent to \eqref{eq:D-heat-inibdry}) in the sense of 
\begin{gather*}
\Big\|\sqrt{a(\cdot)}\Big( u(\cdot,t)-v(\cdot,t)\Big)\Big\|_{L^2(\Omega)}
=O(t^{-\frac{N-\alpha}{2(2-\alpha)}-\eta}), 
\\
\big\|\sqrt{a(\cdot)}v(\cdot,t)\big\|_{L^2(\Omega)}=O(t^{-\frac{N-\alpha}{2(2-\alpha)}})
\end{gather*}
as $t\to \infty$ for some $\eta>0$
(for detail, see \cite{ChHa03, HsLi92, Ik02, Kar00, Ni03MathZ, Nis16, RaToYo16, Wa14JHDE, YaMi00}).
Recently in Sobajima--Wakasugi \cite{SoWa19_CCM}, 
diffusion phenomena for the slowly decaying initial data
has been proved with $a(x)=|x|^{-\alpha}$; 
note that this result is only valid for the damping 
with special structure of homogeneous polynomial type. 
The diffusion phenomena for general damping satisfying the behavior 
at the spatial infinity \eqref{ass:a_intro} is open so far. 

The consequence of the second application of supersolutions in Theorem \ref{thm:main} is diffusion phenomena for \eqref{eq:dw} 
under the assumption \eqref{ass:a_intro}. 

Before stating the result of diffusion phenomena, 
we provide weighted energy estimates for \eqref{eq:dw}.  
\begin{theorem}\label{thm:wee}
Assume that $a(x)$ satisfies \eqref{ass:a_intro} with 
$\alpha\in [0,1)$
and 
the pair $(u_0,u_1)\in H^1_0(\Omega)\times L^2(\Omega)$ satisfies
\[
E_0
=
\int_{\Omega}\Big(|\nabla u_0(x) |^2+ \big(u_1(x)\big) ^2\Big)\lr{x}^{2\sigma+\alpha}\,dx
+
\int_{\Omega} \big(u_0(x)\big)^2\lr{x}^{2\sigma-\alpha}\,dx<\infty.
\]
with $\sigma\in (0,\frac{N-\alpha}{2})$.
Let $u$ be a solution of \eqref{eq:dw} (in a weak sense). 
Then there exists a positive constant $C$ such that 
\begin{align}
\label{eq:wee.1}
\sup_{t\geq 0}
\int_{\Omega}\Big(|\nabla u(x,t)|^2+\big(\pa_tu(x,t)\big)^2\Big)
\Big(1+t+\lr{x}^\alpha\Big)\Big(1+t+\lr{x}^{2-\alpha}\Big)^{\frac{2\sigma}{2-\alpha}}\,dx
&\leq CE_0, 
\\
\label{eq:wee.2}
\sup_{t\geq 0}
\int_{\Omega}\big(u(x,t)\big)^2
\lr{x}^{-\alpha}\Big(1+t+\lr{x}^{2-\alpha}\Big)^{\frac{2\sigma}{2-\alpha}}\,dx
&\leq CE_0, 
\\
\label{eq:wee.3}
\int_{0}^\infty
\left(\int_{\Omega}\big(\pa_tu(x,t)\big)^2\lr{x}^{-\alpha}\Big(1+t+\lr{x}^\alpha\Big)\Big(1+t+\lr{x}^{2-\alpha}\Big)^{\frac{2\sigma}{2-\alpha}}\,dx\right)\,dt
&\leq CE_0.
\end{align}
\end{theorem}
From the estimates \eqref{eq:wee.1}, we already have 
the energy decay estimate 
\begin{equation}\label{wee.weak}
(1+t)^{1+\frac{2\sigma}{2-\alpha}}\int_{\Omega}\Big(|\nabla u(x,t)|^2+\big(\pa_tu(x,t)\big)^2\Big)\,dx
\leq CE_0.
\end{equation}
Combining \eqref{eq:wee.3} with the usual energy equality
\begin{equation*}
\int_{\Omega}\Big(|\nabla u(x,t)|^2+\big(\pa_tu(x,t)\big)^2\Big)\,dx
=
2\int_t^\infty\left(\int_{\Omega}a(x)\big(\pa_tu(x,s)\big)^2\,dx\right)\,ds
\end{equation*}
(verified from $\int_{\Omega}\big(|\nabla u|^2+(\pa_tu)^2\big)\,dx\to 0$ as $t\to \infty$), we obtain an energy decay estimate which is slightly stronger than \eqref{wee.weak}.
\begin{corollary}\label{thm:energy-decay}
Under the assumption of Theorem \ref{thm:wee}, one has
\[
\lim_{t\to \infty}
	\left(
		(1+t)^{1+\frac{2\sigma}{2-\alpha}}
		\int_{\Omega}\Big(|\nabla u(x,t)|^2
		+\big(\pa_tu(x,t)\big)^2\Big)\,dx
	\right)
	=0.
\]
\end{corollary}

The following assertion describes diffusion phenomena 
for \eqref{eq:dw} with polynomially decaying initial data. 
\begin{theorem}\label{thm:dif-pheno}
Assume that $a(x)$ satisfies \eqref{ass:a_intro} and 
the pair $(u_0,u_1)\in (H^2(\Omega)\cap H^1_0(\Omega))\times H^1_0(\Omega)$ satisfies
\begin{align*}
E_0
&=
\int_{\Omega}\Big(|\nabla u_0 (x) |^2+ \big(u_1(x)\big)^2\Big)\lr{x}^{2\sigma+\alpha}\,dx
+
\int_{\Omega} \big(u_0(x)\big)^2\lr{x}^{2\sigma-\alpha}\,dx<\infty, 
\\
E_0'
&=
\int_{\Omega}\Big(|\nabla u_1(x)|^2+\big(u_2(x)\big)^2\Big)\lr{x}^{2\sigma+3\alpha}\,dx
+
\int_{\Omega}(u_1(x))^2\lr{x}^{2\sigma+\alpha}\,dx<\infty
\end{align*}
with $u_2(x)=-\Delta u_0(x)+a(x)u_1(x)$ and $\sigma\in (0,\frac{N-\alpha}{2})$. 
Let $u$ and $v$ be solutions of \eqref{eq:dw} and \eqref{eq:dw-heat} with $v_0(x)=u_0(x)+a(x)^{-1}u_1(x)$, respectively. 
Then there exists a positive constant $K$ such that 
\[
\Big\|\sqrt{a(\cdot)}\Big(u(x,t)-v(x,t)\Big)\Big\|_{L^2(\Omega)}
\leq 
K(1+t)^{-\frac{\sigma}{2-\alpha}}\eta(t)\sqrt{E_0+E_0'},
\]
where 
\begin{align*}
	\eta (t) = 
	\begin{cases}
		(1+t)^{-\frac{2(1-\alpha)}{2-\alpha}} 
		\sqrt{\log (2+t)}
		&\text{if}\ \sigma\in [\alpha,\frac{N-\alpha}{2}),
\\
		(1+t)^{-\frac{2(1-\alpha)\sigma}{(2-\alpha)\alpha}}
&\text{if}\ \sigma\in (0,\alpha).
		\end{cases}
\end{align*}
\end{theorem}

\begin{remark}
Since $v_0=u_0+a(x)^{-1}u_1$ (the initial value of $v$) satisfies $\lr{x}^{\sigma-\frac{\alpha}{2}}v_0\in L^2(\Omega)$, 
Proposition \ref{prop:heat-decay} gives 
\begin{align}
\label{eq:principle-term}
\big\|\sqrt{a(\cdot)}v(t)\big\|_{L^2(\Omega)}\leq C(1+t)^{-\frac{\sigma}{2-\alpha}}
\big\|\lr{x}^{\sigma-\frac{\alpha}{2}}v_0\big\|_{L^2(\Omega)}.
\end{align}
Therefore Theorem \ref{thm:dif-pheno} enables us 
to conclude that $v(t)$ provides the asymptotic profile of the solution $u$.  
In \cite[Remark 1.4]{SoWa19_CCM}, a decay estimate similar to \eqref{eq:principle-term} 
via weighted $L^p$-$L^q$ estimates are provided but with small extra growth factor $(1+t)^{\ep}$. 
The merit of this procedure via Proposition \ref{prop:heat-decay} is 
to reduce unexpected factor $(1+t)^\ep$. 
\end{remark}

As a corollary of Theorem \ref{thm:dif-pheno}, 
we also have a limiting case $\sigma=\frac{N-\alpha}{2}$ as follows.  

\begin{corollary}\label{cor:dif-pheno2}
Assume that $a(x)$ satisfies \eqref{ass:a_intro} and 
the pair $(u_0,u_1)\in (H^2(\Omega)\cap H^1_0(\Omega))\times H^1_0(\Omega)$ satisfies
\begin{align*}
\lr{x}^{\frac{N}{2}}u_0, 
\lr{x}^{\frac{N}{2}+\alpha}\nabla u_0, 
\lr{x}^{\frac{N}{2}+\alpha}u_1, 
\lr{x}^{\frac{N}{2}+2\alpha}\Delta u_0, 
\lr{x}^{\frac{N}{2}+2\alpha}\nabla u_1
\in L^2(\Omega). 
\end{align*}
Let $u$ and $v$ be solutions of \eqref{eq:dw} and \eqref{eq:dw-heat} with $v_0(x)=u_0(x)+a(x)^{-1}u_1(x)$, respectively. 
Then for every $\ep>0$, 
there exists a positive constant $K_\ep$ such that 
\[
\Big\|\sqrt{a(\cdot)}\Big( u(\cdot,t)-v(\cdot,t) \Big)\Big\|_{L^2(\Omega)}
\leq 
K_\ep(1+t)^{-\frac{N-\alpha}{2(2-\alpha)}-\frac{2(1-\alpha)}{2-\alpha}+\ep}.
\]
\end{corollary}

\begin{remark}
Energy estimates with polynomial growth weights (such as Theorem \ref{thm:wee}) can be also applied to 
semilinear wave equations with damping term $a(x)\pa_tu$ (for example, see Sobajima \cite{So18a} for the case $a(x)\equiv 1$). 
This kind of analysis including asymptotic behavior of solutions to nonlinear problem will be done in a forthcoming paper.
\end{remark}

The present paper is organized as follows. 
In Section \ref{sec:pre}, 
the basic properties of Kummer's confluent hypergeometric functions are collected, 
which are deeply used throughout of this paper. 
Section \ref{sec:supersol} 
is devoted to prove Theorem \ref{thm:main}, that is, 
the construction of supersolution to $\pa_tU-D(x)\Delta U=0$.
As applications, 
we will prove weighted $L^2$-type decay estimates
for initial-boundary problem of $\pa_tv-D(x)\Delta v=0$
via semigroup approach in Section \ref{sec:heat}.
In the last Section \ref{sec:damped}, 
we show weighted energy estimates 
for solutions to wave equations with damping term 
$\pa_t^2u-\Delta u+a(x)\pa_tu=0$ and their diffusion phenomena.

\section{Preliminaries}\label{sec:pre}

We collect some of important properties of 
Kummer's confluent hypergeometric functions. 
At the beginning we state their definition. 
\begin{definition}[Kummer's confluent hypergeometric functions]
\label{def:M}
For
$b,c \in \mathbb{R}$ with $-c \notin \mathbb{N} \cup \{0\}$,
Kummer's confluent hypergeometric function of first kind is defined by
\begin{align*}
	M(b, c; s) = \sum_{n=0}^{\infty} \frac{(b)_n}{(c)_n} \frac{s^n}{n!}, \quad s\in [0,\infty),
\end{align*}
where $(d)_n$ is the Pochhammer symbol defined by 
$(d)_0 = 1$ and $(d)_n = \prod_{k=1}^n (d+k-1)$ for $n\in\N$; 
note that when $b=c$, Kummer's function $M(b,b;s)$ coincides with $e^s$.
\end{definition}

The following properties of Kummer's confluent hypergeometric functions 
are well known (see e.g., Beals--Wong \cite{BeWo10}). 
\begin{lemma}\label{lem_Kummer}
Kummer's confluent hypergeometric function
$M(b, c; s)$ satisfies the following properties:
\begin{itemize}
\item[{\rm (i)}]
$M(b, c; s)$ is a solution of
the Kummer equation
\begin{align}
\label{Kummer_eq}
	s u(s)'' + (c-s) u'(s) - bu(s) = 0.
\end{align}
\item[{\rm (ii)}]
If $c \ge b > 0$, then
$M(b, c; s) > 0$
and
$M(b, c; s) \sim \frac{\Gamma(c)}{\Gamma(b)} s^{b-c} e^s$
as
$s\to \infty$,
more precisely,
\begin{align*}
	\lim_{s\to \infty} \frac{M(b, c; s)}{s^{b-c}e^s} = \frac{\Gamma(c)}{\Gamma(b)}.
\end{align*}
\item[{\rm (iii)}]
More generally, if $-c \notin \mathbb{N}\cap \{0\}$ and $c \ge b$, then,
while the sign of
$M(b, c; s)$
is indefinite,
it has the same asymptotic behavior
\begin{align*}
	\lim_{s\to \infty} \frac{M(b, c; s)}{s^{b-c}e^s} = \frac{\Gamma(c)}{\Gamma(b)}.
\end{align*}
In particular, 
$M(b, c; s)$
has a bound
\begin{align*}
	| M(b,c;s) | \le C_{b,c} (1+s)^{b-c} e^s
\end{align*}
with some constant $C_{b,c}>0$.
\item[{\rm (iv)}]
$M(b, c; s)$
satisfies the relations
\begin{align*}%
	s M(b,c;s) &= sM'(b,c;s) + (c-b)M(b,c;s) - (c-b)M(b-1,c;s),\\
	c M'(b,c;s) &= cM(b,c;s) - (c-b) M(b,c+1;s).
\end{align*}%
\end{itemize}
\end{lemma}
\begin{remark}
For the proof of (i), (ii), and (iv), see \cite[p.190, (6.1.1)]{BeWo10}, \cite[p.192, (6.1.8)]{BeWo10},
and \cite[p.200]{BeWo10}, respectively.
The assertion (iii) can be found in \cite[p.192, the comment under (6.1.9)]{BeWo10}.
Also, we easily prove it in the following way.
We first note that
\begin{align*}
	M^{(m)}(b,c;s) = \frac{(b)_{m}}{(c)_m} M(b+m, c+m; s)
\end{align*}
and
$|M^{(m)}(b,c;s)| \to \infty$ as $s \to \infty$ for any $m \in \mathbb{Z}_{\ge 0}$.
Let
$m \in \mathbb{Z}_{\ge 0}$
be such that $b+m > 0$ holds.
Then, the l'H\^{o}pital theorem and (ii) imply
\begin{align*}
	\lim_{s \to \infty} \frac{M(b,c;s)}{s^{b-c}e^s}
	&= \lim_{s \to \infty} \frac{D_s^m M(b,c;s)}{D_s^m (s^{b-c}e^s)}
	= \frac{(b)_{m}}{(c)_m} \lim_{s \to \infty} \frac{M(b+m, c+m; s)}{s^{b-c}e^s + o(s^{b-c}e^s)}
	= \frac{(b)_{m}}{(c)_m} \frac{\Gamma(c+m)}{\Gamma(b+m)} = \frac{\Gamma(c)}{\Gamma(b)}.
\end{align*}
\end{remark}

\section{Construction of supersolution to $\pa_tU=D(x)\Delta U$}\label{sec:supersol}

In this section, we will construct a family of positive supersolutions to the 
parabolic equation
\begin{align}\label{eq:sect3}
\pa_tU-D(x)\Delta U=0, \quad (x,t) \in \Omega\times (0,\infty).
\end{align}
Here we assume that $\alpha\in (-\infty,\min\{2,N\})$ 
in \eqref{eq:D-ass}.
Following the previous works \cite{SoWa16_JDE, SoWa17_AIMS, SoWa18_ADE}, 
we use the same notation as the damping coefficient $a(x)$. Therefore we put 
\begin{align}%
\label{a}
	a(x) := \frac{1}{D(x)}.
\end{align}%
Since
$D(x)$
is positive, the problem
\eqref{eq:D-heat} is equivalent to
\begin{align}\label{eq:a-heat}
	a(x) \pa_t U (x,t)-\Delta U(x,t)=0, 
\quad (x,t)\in \Omega\times (0,\infty).
\end{align}
Also, the assumption \eqref{eq:D-ass} implies that
\begin{align}\label{ass.a.sec3}%
	\lim_{|x|\to \infty} \left( |x|^{\alpha} a(x) \right) = \frac{1}{D_0} > 0.
\end{align}%

To fix the direction of the discussion, 
we first recall the special case where $a(x)=|x|^{-\alpha}$ ($D(x) = |x|^{\alpha}$)
with $\alpha\in [0,1)$, 
which is studied in \cite[Section 2]{SoWa19_CCM}. 
In this case, the equation \eqref{eq:sect3} becomes
\begin{align}\label{eq:sect3.spec}
	|x|^{-\alpha}\pa_tU = \Delta U,\quad (x,t)\in \R^N\times (0,\infty).
\end{align}
This equation has a self-similar structure. 
Indeed, if $U(x,t)$ is the solution of \eqref{eq:sect3.spec}, then $U_s(x,t)=U(s^{\frac{1}{2-\alpha}}x,st)$ 
is also a solution of the same equation. 
Therefore, we can introduce a notion of self-similar solutions of
the type $U=s^{\beta}U_s$ $(\beta>0)$. 
Fortunately, such a family of solutions can be explicitly written by using Kummer's confluent hypergeometric functions. 
(In this moment, to use this kind of self-similar structure we shall impose $\alpha<2$. 
The other cases can be considered via the Kelvin transform.)
The following lemma is the list of their important properties. 

\begin{lemma}[{\cite[Section 2]{SoWa19_CCM}}]\label{lem:exact}
For $\beta>0$ and $\alpha\in [0,1)$, define 
\[
\widetilde{\Phi}_{\beta}(x,t)
=t^{-\beta}\varphi_{\beta}\big(\xi(x,t)\big), 
\quad 
\varphi_{\beta}(z)=e^{-z}M\left(\frac{N-\alpha}{2-\alpha}-\beta,\frac{N-\alpha}{2-\alpha};z\right), 
\quad
\xi(x,t)=\frac{|x|^{2-\alpha}}{(2-\alpha)^2t}.
\]
Then $\widetilde{\Phi}_{\beta}\in C^\infty(\R^N\times (0,\infty))$ satisfies the following assertions:
\begin{itemize}
\item[(i)]
for every $\beta>0$ and $s>0$, 
\[
\widetilde{\Phi}_{\beta}(x,t)=s^{\beta}\widetilde{\Phi}_{\beta}(s^{\frac{1}{2-\alpha}}x,st), 
\quad (x,t)\in \R^N\times (0,\infty);
\]
\item[(ii)]
for every $\beta>0$, 
\[
|x|^{-\alpha}\pa_t\widetilde{\Phi}_{\beta}(x,t)=\Delta \widetilde{\Phi}_{\beta}(x,t), 
\quad (x,t)\in \R^N\times (0,\infty);
\]
\item[(iii)]
for every $\beta>0$, there exists a positive constant $C_\beta >0$ such that 
\[
|\widetilde{\Phi}_{\beta}(x,t)|\leq C_\beta\left(t+\frac{|x|^{2-\alpha}}{(2-\alpha)^2}\right)^{-\beta}, 
\quad (x,t)\in \R^N\times (0,\infty);
\]
\item[(iv)]
for every $0<\beta<\frac{N-\alpha}{2-\alpha}$, 
there exists a positive constant $c_\beta>0$ such that
\[
\widetilde{\Phi}_{\beta}(x,t)\geq c_\beta\left(t+\frac{|x|^{2-\alpha}}{(2-\alpha)^2}\right)^{-\beta}, 
\quad (x,t)\in \R^N\times (0,\infty);
\]
\item[(v)]
for every $\beta>0$, 
\[
\pa_t\widetilde{\Phi}_{\beta}(x,t)=-\beta \widetilde{\Phi}_{\beta+1}(x,t), 
\quad (x,t)\in \R^N\times (0,\infty).
\]
\end{itemize}
\end{lemma}

The aim of this section is to construct 
a family of supersolutions to \eqref{eq:sect3} 
of the form 
\[
(t_0+t)^{-\beta}\varphi\left( \frac{\gamma A(x)}{t_0+t}\right)
\]
with properties similar to Lemma \ref{lem:exact} (iii)--(v), 
which is motivated by the family $\{\widetilde{\Phi}_\beta\}_{\beta>0}$ in Lemma \ref{lem:exact}.  
\subsection{Related elliptic problem}
Observe that the relation 
\[
\Delta \left(\frac{|x|^{2-\alpha}}{(2-\alpha)^2}\right)=\frac{N-\alpha}{2-\alpha}|x|^{-\alpha}
=({\rm const.})a(x)
\]
seems to be important.
As the generalization of this relation, 
we recall existence of {\it approximate} solutions of the Poisson equation
\begin{align}\label{eq:ell}%
	\Delta A(x) = a(x),\quad x\in \mathbb{R}^N
\end{align}%
in the sense \eqref{A1} below. 
We refer {\cite[Lemma 2.1]{SoWa17_AIMS} for $\alpha\in [0,1)$}, {\cite[Lemma 3.1]{SoWa18_ADE} for $\alpha\in (-\infty,0)$}
for the proof. 
The proof of the remaining cases $\alpha\in [1,2)$ 
is essentially the same as  {\cite[Lemma 2.1]{SoWa17_AIMS}.
\begin{lemma}\label{lem_A_ep}
Assume that $a(x)$ satisfies \eqref{ass.a.sec3} with 
$\alpha\in (-\infty,\min\{2,N\})$. 
Then for every $\ep \in (0,1)$, 
there exist a function $A_\ep\in C^2(\R^N)$ 
and positive constants $c_\ep$ and $C_\ep$ such that
\begin{align}
\label{A1}
	&(1-\ep)a(x)\leq \Delta A_\ep (x) \leq (1+\ep) a(x),\\
\label{A2}
	&c_\ep \lr{x}^{2-\alpha} \leq A_\ep(x) \leq C_\ep \lr{x}^{2-\alpha},\\
\label{A3}
	&\frac{|\nabla A_\ep(x)|^2}{a(x)A_\ep(x)}\leq \frac{2-\alpha}{N-\alpha}+\ep
\end{align}
hold for $x\in \mathbb{R}^N$.
\end{lemma}
\begin{remark}
It is enough to find an exact solution of \eqref{eq:ell}
satisfying \eqref{A2} and \eqref{A3}
to construct a supersolution of parabolic equation \eqref{eq:sect3}. 
If $a(x)$ is radially symmetric, 
the solution $A$ of Poisson equation was constructed 
in Todorova--Yordanov \cite{ToYo09}. 
In subsequent papers, 
suitable profiles for more general problem, 
such as damped wave equation with 
space-dependent diffusion, have been considered 
(see e.g., 
Radu--Todorova--Yordanov \cite{RaToYo09,RaToYo10}).
However, such a solution with expected properties 
does not exist in general (especially in the case where $a(x)$ is not radially symmetric, see \cite[Remark 3.1]{SoWa16_JDE}). 
\end{remark}
\subsection{Supersolution of corresponding parabolic equation}

For a function
$\varphi \in C^2([0,\infty))$
and
$(x,t) \in \mathbb{R}^N \times [0,\infty)$,
we put 
\begin{align}\label{pre-Phi}
\Phi(x,t)
=(t_0+t)^{-\beta}
\varphi\left(z\right), \quad z=\frac{\gamma A_\ep(x)}{t_0+t}
\end{align}
with some constants $\beta,\gamma>0$ and $t_0 \ge 1$.
Here,
$A_\ep(x)$
is the function constructed in Lemma \ref{lem_A_ep}
with a constant
$\varepsilon \in (0,1)$.
We will specify the function
$\varphi$
later, and here we first show
the following lemma. 
\begin{lemma}\label{lem.3.3}
Let $\varphi\in C^2([0,\infty))$ and 
$\Phi$ be as in \eqref{pre-Phi}. Then 
\begin{equation}\label{eq:Phi}%
	a(x)\pa_t \Phi(x,t)-\Delta \Phi(x,t)
	=
	-a(x)(t_0+t)^{-\beta-1}
		\left(
			\beta \varphi(z)
			+z \varphi'(z)
			+\gamma\,\frac{\Delta A_\ep(x)}{a(x)} \varphi'(z)
			+\gamma\, \frac{|\nabla A_\ep(x)|^2}{a(x)A_\ep(x)}z \varphi''(z)
		\right)
\end{equation}%
holds for
$(x,t) \in \mathbb{R}^N \times [0,\infty)$.
\end{lemma}
\begin{proof}
By direct calculation we have 
\begin{align*}
	a(x)\pa_t \Phi(x,t)
	&=a(x)
	(t_0+t)^{-\beta-1}
	\left(
	-
	\beta 
	\varphi\left(z\right)
	-
	z
	\varphi'\left(z\right)
	\right),
\end{align*}
where
$z=\gamma A_\ep(x)/(t_0+t)$.
On the one hand, we see that 
\begin{align*}
\Delta \Phi(x,t)
&=
(t_0+t)^{-\beta}
\left(\gamma\,\frac{\Delta A_\ep(x)}{t_0+t}
\varphi'\left(z\right)
+
\gamma^2\,\frac{ |\nabla A_\ep(x)|^2}{(t_0+t)^2}
\varphi''\left(z\right)
\right)
\\
&=
a(x)(t_0+t)^{-\beta-1}
\left(\gamma \frac{\Delta A_\ep(x)}{a(x)}
\varphi'\left(z\right)
+
\frac{\gamma|\nabla A_\ep(x)|^2}{a(x) A_\ep(x)}
z\varphi''\left(z\right)
\right).
\end{align*}
Therefore we have \eqref{eq:Phi}.
\end{proof}

In particular, we choose 
\begin{align}
\label{gammatilde}
	\widetilde{\gamma}_\ep=\left(\frac{2-\alpha}{N-\alpha}+\ep\right)^{-1}, \quad 
	\gamma_\ep=(1-\ep)
	\widetilde{\gamma}_\ep.
\end{align}
and 
$\varphi$ by using the Kummer confluent hypergeometric function.
\begin{definition}\label{phi.beta}
For $\beta\in \R$, define 
\[
	\varphi_{\beta,\ep}(s)=e^{-s}M\left(\gamma_\ep-\beta, \gamma_\ep; s\right),
	\quad s\geq 0.
\]
We remark that $\varphi_{\beta,\ep}$ is a unique (modulo constant multiple)
solution of
the equation 
\begin{align}
\label{eq:varphi}
	s\varphi''(s)+(\gamma_\ep+s)\varphi'(s)+\beta\varphi(s)=0 
\end{align}
with bounded derivative near $s=0$. 
\end{definition}
Then we have the following properties of $\varphi_{\beta,\ep}$. 

\begin{lemma}\label{lem_varphi}
The function
$\varphi_{\beta,\ep}$
satisfies the following:
\begin{itemize}
\item[\rm (i)]
If $\beta \in [0,\gamma_\ep)$,
then
$\varphi_{\beta,\ep}$
has the lower and upper bounds
\[
	k_{\beta,\ep}(1+s)^{-\beta}\leq \varphi_{\beta,\ep}(s)\leq K_{\beta,\ep}(1+s)^{-\beta}
\]
with some constants
$k_{\beta,\varepsilon}, K_{\beta,\varepsilon} > 0$.
\item[{\rm (ii)}]
If $\beta \ge 0$,
then
$\varphi_{\beta,\ep}$
has the upper bound
\begin{align*}
	|\varphi_{\beta,\varepsilon}(s)|
		\le K_{\beta,\varepsilon} (1+s)^{-\beta}
\end{align*}
with some constant
$K_{\beta,\varepsilon} > 0$.
\item[\rm (iii)]
If
$\beta \ge 0$,
then
$\varphi_{\beta,\ep}(s)$
and
$\varphi_{\beta+1,\ep}(s)$
satisfy the recurrence relation
\[
	\beta\varphi_{\beta,\ep}(s)
	+
	s\varphi_{\beta,\ep}'(s)
	=
	\beta\varphi_{\beta+1,\ep}(s).
\]
\item[\rm (iv)]
If $0<\beta <\gamma_\ep$,
the first and the second derivatives of
$\varphi_{\beta,\ep}(s)$
have negative and positive signs,
respectively:
\[
	\varphi_{\beta,\ep}'(s)
	=-\frac{\beta}{\gamma_\ep} e^{-s}M(\gamma_\ep-\beta,\gamma_\ep+1;s)<0, 
	\quad 
	\varphi_{\beta,\ep}''(s)
	=\frac{\beta(\beta+1)}{\gamma_\ep(\gamma_\ep+1)} e^{-s}M(\gamma_\ep-\beta,\gamma_\ep+2;s)>0.
\]
\end{itemize}
\end{lemma}

\begin{proof}
(i) From Lemma \ref{lem_Kummer} (ii), we obtain
\[
	k_{\beta, \varepsilon} e^s (1+s)^{-\beta}
	\le M\left(\gamma_\ep-\beta, \gamma_\ep; s\right)
	\le K_{\beta, \varepsilon} e^s (1+s)^{-\beta}
\]
for
$s\ge 0$
with some positive constants
$k_{\beta, \varepsilon}$, $K_{\beta, \varepsilon}$.
This implies (i).

\noindent
(ii) Similarly to (i), Lemma \ref{lem_Kummer} (iii) implies
\begin{align*}
	|M\left(\gamma_\ep-\beta, \gamma_\ep; s\right)|
		\le K_{\beta,\varepsilon} (1+s)^{-\beta} e^s.
\end{align*}
This and the definition of
$\varphi_{\beta,\varepsilon}$
lead to (ii).

\noindent
(iii)
Noting
\begin{align}%
\label{varphi'}
	\varphi_{\beta,\ep}'(s)
	= e^{-s} \left( - M(\gamma_{\ep} - \beta, \gamma_{\ep}; s) + M'(\gamma_{\ep} - \beta, \gamma_{\ep}; s) \right)
\end{align}%
and the first assertion of Lemma \ref{lem_Kummer} (iv),
we have
\begin{align*}%
	&\beta\varphi_{\beta,\ep}(s)
	+
	s\varphi_{\beta,\ep}'(s) \\
	&=
		e^{-s} \left( \beta M(\gamma_{\ep}-\beta, \gamma_{\ep}; s)
			-s M(\gamma_{\ep}-\beta, \gamma_{\ep}; s)
			+s M'(\gamma_{\ep}-\beta, \gamma_{\ep}; s)
				\right) \\
	&=
		\beta e^{-s} M(\gamma_{\ep}-\beta-1, \gamma_{\ep}; s) \\
	&= \beta\varphi_{\beta+1,\ep}(s).
\end{align*}%

\noindent
{\rm (iv)}
The second assertion of Lemma \ref{lem_Kummer} (iv) implies
\begin{align*}%
	\gamma_{\ep} M'(\gamma_{\ep}-\beta, \gamma_{\ep}; s)
	&= \gamma_{\ep} M(\gamma_{\ep}-\beta, \gamma_{\ep}; s)
		- \beta M(\gamma_{\ep}-\beta, \gamma_{\ep}+1; s).
\end{align*}%
From this and \eqref{varphi'}, we obtain
\begin{align*}%
	\varphi_{\beta,\ep}'(s)
	=-\frac{\beta}{\gamma_\ep} e^{-s}M(\gamma_\ep-\beta,\gamma_\ep+1;s).
\end{align*}%
Since $0< \beta < \gamma_{\ep}$,
Lemma \ref{lem_Kummer} (ii) shows
$M(\gamma_\ep-\beta,\gamma_\ep+1;s) > 0$
and hence,
$\varphi_{\beta,\ep}'(s) < 0$.

Next, we compute
\begin{align*}%
	\varphi_{\beta,\ep}''(s)
	&=-\frac{\beta}{\gamma_\ep} e^{-s}
		\left( -M(\gamma_\ep-\beta,\gamma_\ep+1;s) + M'(\gamma_\ep-\beta,\gamma_\ep+1;s) \right).
\end{align*}%
Applying the second assertion of Lemma \ref{lem_Kummer} (iv), we have
\begin{align*}%
	(\gamma_{\ep}+1) M'(\gamma_{\ep}-\beta, \gamma_{\ep}+1; s)
	&= (\gamma_{\ep}+1) M(\gamma_{\ep}-\beta, \gamma_{\ep}+1; s)
		- (\beta+1) M(\gamma_{\ep}-\beta, \gamma_{\ep}+2; s),
\end{align*}%
and hence,
\begin{align*}%
	\varphi_{\beta,\ep}''(s)
	&=\frac{\beta(\beta+1)}{\gamma_\ep(\gamma_{\ep}+1)} e^{-s}
		M(\gamma_{\ep}-\beta, \gamma_{\ep}+2; s).
\end{align*}%
Noting again that
$0<\beta < \gamma_{\ep}$,
we see from Lemma \ref{lem_Kummer} (ii) that
$M(\gamma_{\ep}-\beta, \gamma_{\ep}+2; s) > 0$,
and hence,
$\varphi_{\beta,\ep}''(s) > 0$.
\end{proof}

Here we define a family of functions $\{\Phi_{\beta,\ep}\}_{\beta}$ which will be proved to fulfill all conditions in 
Theorem \ref{thm:main}. 

\begin{definition}\label{phi.beta.ep}
For
$(x,t) \in \mathbb{R}^N \times [0,\infty)$,
we define 
\[
	\Phi_{\beta,\ep}(x,t; t_0)=(t_0+t)^{-\beta}\varphi_{\beta,\ep}(z), 
	\quad 
	z=\frac{\widetilde{\gamma}_\ep A_\ep(x)}{t_0+t},
\]
where
$\ep \in (0,1)$,
$\widetilde \gamma_{\ep}$ is the constant given in \eqref{gammatilde},
$\beta$ is a constant satisfying $\beta \in (0,\widetilde{\gamma_{\ep}})$,
$t_0 \ge 1$,
$\varphi_{\beta,\ep}$
is the function defined by Definition \ref{phi.beta},
and
$A_{\ep}(x)$ is the function constructed in Lemma \ref{lem_A_ep}.
\end{definition}

The function
$\Phi_{\beta,\ep}(x,t)$
defined above is a supersolution of the equation \eqref{eq:a-heat}.

\begin{lemma}\label{lem:super-sol}
The function
$\Phi_{\beta,\ep}(x,t;t_0)$
satisfies
\begin{align*}
	a(x)\pa_t\Phi_{\beta,\ep}(x,t;t_0)-\Delta \Phi_{\beta,\ep}(x,t;t_0)\geq 0,\quad
	(x,t) \in \mathbb{R}^N \times [0,\infty).
\end{align*}
\end{lemma}
\begin{proof}
We note that the equation \eqref{eq:varphi} implies
\begin{align*}%
	\beta \varphi_{\beta,\ep}(z)
	+z \varphi_{\beta,\ep}'(z)
	&=
		- \gamma_{\ep} \varphi_{\beta,\ep}'(z)
		- z \varphi_{\beta,\ep}''(z) \\
	&= - (1-\ep) \widetilde{\gamma_{\ep}} \varphi_{\beta,\ep}'(z)
		- z \varphi_{\beta,\ep}''(z).
\end{align*}%
Using the above and Lemma \ref{lem.3.3}, we calculate
\begin{align*}
&a(x)\pa_t \Phi_{\beta,\ep}(x,t)-\Delta \Phi_{\beta,\ep}(x,t)
\\
&=
-a(x)(t_0+t)^{-\beta-1}
\left(
\beta \varphi_{\beta,\ep}(z)
+z \varphi_{\beta,\ep}'(z)
+\widetilde{\gamma}_\ep\,\frac{\Delta A_\ep(x)}{a(x)} \varphi_{\beta,\ep}'(z)
+\widetilde{\gamma}_\ep\, \frac{| \nabla A_\ep(x)|^2}{a(x)A_\ep(x)}z \varphi_{\beta,\ep}''(z)
\right)
\\
&=
\widetilde{\gamma}_\ep a(x)(t_0+t)^{-\beta-1}
\left(
1-\ep-\frac{\Delta A_\ep(x)}{a(x)} 
\right)\varphi_{\beta,\ep}'(z)
+a(x)(t_0+t)^{-\beta-1}
z\left(
1-\widetilde{\gamma}_\ep\, \frac{| \nabla A_\ep(x)|^2}{a(x)A_\ep(x)}
\right)\varphi_{\beta,\ep}''(z).
\end{align*}
Since it follows from the construction of $A_\ep$ that 
\[
1-\ep-\frac{\Delta A_\ep(x)}{a(x)}\leq 0, 
\quad
1-\widetilde{\gamma}_{\ep}\,\frac{| \nabla A_\ep(x)|^2}{a(x)A_\ep(x)}\geq 0,
\]
we see from Lemma \ref{lem_varphi} {\rm (iii)} that $\Phi_{\beta,\ep}$ satisfies 
$a(x)\pa_t\Phi_{\beta,\ep}(x,t)-\Delta \Phi_{\beta,\ep}(x,t)\geq 0$
for all $x\in \mathbb{R}^N$
and
$t\geq 0$. 
\end{proof}

The function $\Phi_{\beta,\varepsilon}$ also satisfies
the following recurrence relation.
\begin{lemma}\label{lem.dtphi}
Let $\Phi_{\beta,\varepsilon}$ be defined in Definition \ref{phi.beta.ep}.
Then, for $\beta \ge 0$, we have
\begin{align*}
	\partial_t \Phi_{\beta,\varepsilon}(t;t_0) = -\beta \Phi_{\beta+1,\varepsilon}(t;t_0).
\end{align*}
\end{lemma}
The lemma above immediately follows from Lemma \ref{lem_varphi} (iii)
and we omit the detail.

\begin{proof}[Proof of Theorem \ref{thm:main}]
Put $\beta=\frac{2\sigma}{2-\alpha}$ 
and take $\ep>0$ sufficiently small so that $\beta<\widetilde{\gamma}_{\ep}$, and let $t_0\geq 1$. We define
\[
\overline{U}_{D,\sigma}(x,t)=\widetilde{K} \Phi_{\beta,\ep}(x,t;t_0), \quad (x,t)\in \Omega\times [0,\infty),
\]
where $\widetilde{K}$ is a positive constant specified later. 
Then Lemma \ref{lem:super-sol} yields that 
$\Phi_{\beta,\ep}(x,t;t_0)$ is a supersolution of $a(x)\pa_tU-\Delta U=0$ in $\R^N$. 
Concerning the initial value and the boundary condition, 
we see from Lemma \ref{lem_varphi} (i) 
and Lemma \ref{lem_A_ep} that 
\begin{align}\label{eq:middle}
\widetilde{K}k_{\beta,\ep}
\Big(t_0+t+C_\ep \lr{x}^{2-\alpha}\Big)^{-\frac{2\sigma}{2-\alpha}}
\leq 
\overline{U}_{D,\sigma}(x,t)
\leq 
\widetilde{K}K_{\beta,\ep}
\Big(t_0+t+c_\ep \lr{x}^{2-\alpha}\Big)^{-\frac{2\sigma}{2-\alpha}}.
\end{align}
Taking $\widetilde{K}=k_{\beta,\ep}^{-1}(t_0+C_\ep)^{\frac{2\sigma}{2-\alpha}}$, we have $\overline{U}_{D,\sigma}(x,0)\geq \lr{x}^{-2\sigma}=w(x)$ for $x\in \Omega$. 
Since $\Phi_{\beta,\ep}(x,t;t_0)$ 
is positive in $\R^N\times [0,\infty)$, 
$\overline{U}_{D,\sigma}$ satisfies 
the nonnegativity condition 
on $\pa\Omega\times [0,\infty)$.
As a consequence,
$\overline{U}_{D,\sigma}$ is a supersolution of the problem \eqref{eq:D-heat}. 
The estimate \eqref{eq:two-sided} 
follows from \eqref{eq:middle} and 
the other estimate \eqref{eq:deri-bdd} 
follows from 
Lemma \ref{lem.dtphi} and Lemma \ref{lem_varphi} (ii).  
The proof is now complete. 
\end{proof}

\section{Application to weighted $L^2$-estimates for diffusion equations}\label{sec:heat}

In this section,
as an application of supersolutions to \eqref{eq:D-heat} in Theorem \ref{thm:main}, 
we discuss some weighted $L^2$-estimates for
the initial-boundary value problem of the linear diffusion equation
\eqref{eq:D-heat-inibdry} which is equivalent to 
the following problem in terms of $a(x)=D(x)^{-1}$:
\begin{equation}%
\label{de}
\begin{cases}
a(x)\pa_tu(x,t)- \Delta u(x,t)=0,
&
(x,t)\in \Omega\times (0,\infty),
\\
u(x,t)=0,&
(x,t)\in \pa\Omega\times (0,\infty),
\\
u(0,x)=f(x), 
&
x \in \Omega.
\end{cases}
\end{equation}%
Here,
$\Omega$
is an exterior domain in
$\mathbb{R}^N$ with smooth boundary $\pa\Omega$ 
with $N\ge 2$,
namely,
$\mathbb{R}^N \setminus \Omega$
is compact. The operator $a(x)^{-1}\Delta$ 
is formally symmetric in the following Hilbert space 
$L^2_{d\mu}=L^2_{d\mu}(\Omega)$ with the inner product $(\cdot,\cdot)_{L^2_{d\mu}}$:
\[
L^2_{d\mu}(K):=\left\{f\in L^2_{\rm loc}(K)
\;;\;
\|f\|_{L^2_{d\mu}(K)}=\left(\int_{K}|f|^2\,d\mu\right)^{\frac{1}{2}}<\infty\right\}, \quad 
(f,g)_{L^2_{d\mu}(K)}=\int_{K}fg\,d\mu, \quad d\mu=a(x)dx.
\]
According to the analysis of \cite[Section 2]{SoWa16_JDE},
the corresponding bilinear closed form is given by 
\[
\mathfrak{a}(u,v)=\int_{\Omega}\nabla u\cdot\nabla v\,dx, \quad 
D(\mathfrak{a})=
\left\{u\in L^2_{d\mu}\cap \dot{H}^{1}(\Omega)\;;\;
\int_{\Omega}\frac{\pa u}{\pa x_j}\varphi\,dx=
-\int_{\Omega}u\frac{\pa \varphi}{\pa x_j}, \ \forall \varphi\in C_c^\infty(\R^N)\right\}.
\]
Note that a similar proof to \cite[Lemma 2.1]{SoWa16_JDE} works 
for $\alpha<2$. 
Then we can use the Friedrichs extension $-L$ of $-a(x)\Delta$ in $L^2_{d\mu}$ 
as the associated operator of the closed form $\mathfrak{a}$. 

\begin{lemma}[{\cite[Lemma 2.2]{SoWa16_JDE}}]\label{lem.4.1}
The operator $L$ in $L^2_{d\mu}$ 
defined by 
\begin{align*}
D(L)
&=\left\{u\in D(\mathfrak{a})\;;\;\exists f\in L^2_{d\mu}\ \text{s.t.}\ 
\mathfrak{a}(u,v)=(f,v)_{L^2_{d\mu}}\quad \forall v\in D(\mathfrak{a})
\right\},
\\
-Lu&=f
\end{align*}
is nonnegative and selfadjoint in $L^2_{d\mu}$. Therefore $L$ generates an analytic semigroup $T(t)$ on $L^2_{d\mu}$ and satisfies 
\[
\|T(t)f\|_{L^2_{d\mu}}\leq 
\|f\|_{L^2_{d\mu}}, \quad 
\|LT(t)f\|_{L^2_{d\mu}}
\leq 
\frac{1}{t}\|f\|_{L^2_{d\mu}}
\quad \forall f\in L^2_{d\mu}.
\]
\end{lemma}
Although $L^p$-$L^q$ type 
estimates for the semigroup $T(t)$ 
are proved 
in \cite{SoWa16_JDE} for $\alpha\in (0,1)$ and \cite{SoWa18_ADE} for $\alpha\in (-\infty,0)$, 
we shall provide other type 
decay estimates. 
The main assertion of this section is Proposition \ref{prop:heat-decay} which is rewritten 
in the following way via the notation in this section.  
\begin{proposition}\label{prop.4.2}
Let $\sigma\in (0,\frac{N-\alpha}{2})$. 
If $f\in L^2_{d\mu}$ satisfies $\lr{x}^{\sigma}f\in L^2_{d\mu}$, 
then 
\[
\|T(t)f\|_{L^2_{d\mu}}
\leq C(1+t)^{-\frac{\sigma}{2-\alpha}}\big\|\lr{x}^{\sigma}f\big\|_{L^2_{d\mu}} \quad \forall t\geq 0.
\]
\end{proposition}

\begin{proof}
We introduce the following 
auxiliary problem 
in $\Omega_n=\Omega\cap B(0,n)$, 
which is an 
approximation of 
the original problem \eqref{de}:
\begin{equation}%
\label{de-approx}
\begin{cases}
a(x)\pa_tu(x,t)- \Delta u(x,t)=0,
&
(x,t)\in \Omega_n\times (0,\infty),
\\
u(x,t)=0,&
(x,t)\in \pa\Omega_n\times (0,\infty),
\\
u(0,x)=f(x), 
&
x \in \Omega_n.
\end{cases}
\end{equation}%
It is sufficient to consider only the case where $\pa\Omega \cap \pa B(0,n)= \emptyset$. 
By the same procedure as the case of $L$, we have the corresponding 
generator $L_n$ 
and 
the semigroup $T_n(t)$
satisfying 
$D(L_n)=H^2(\Omega_n)\cap H_0^1(\Omega_n)$ with 
\[
\|T_n(t)f\|_{L^2_{d\mu}(\Omega_n)}\leq 
\|f\|_{L^2_{d\mu}(\Omega_n)}, \quad 
\|L_nT_n(t)g\|_{L^2_{d\mu}(\Omega_n)}
\leq 
\frac{1}{t}\|g\|_{L^2_{d\mu}(\Omega_n)}
\quad \forall g\in L^2_{d\mu}(\Omega_n).
\]
Now we define the semigroup $\{\widetilde{T}_n(t)\}_{t\geq 0}$ 
in $L^2_{d\mu}$ by
\[
\widetilde{T}_n(t)f=
\begin{cases}
w_n(t)=T_n(t)[f|_{\Omega_n}](x) &\quad x\in \Omega_n, 
\\
f(x)&\quad x\in \Omega\setminus \Omega_n.
\end{cases}
\]
Here we choose $\beta=\frac{2\sigma}{2-\alpha}$.
Since 
$\Phi_{\beta,\ep}=\Phi_{\beta,\ep}(\cdot,t;1)\in C^{2}(\overline{\Omega_n})$ is positive, 
we can verify the following computation:
\begin{align*}
\frac{d}{dt}
\int_{\Omega_n} |w_n(t)|^2\Phi_{\beta,\ep}(t;1)^{-1}\,d\mu 
&=
2\int_{\Omega_n} w_n(t)\pa_tw_n(t)\Phi_{\beta,\ep}^{-1}\,d\mu 
-
\int_{\Omega_n} |w_n(t)|^2\Phi_{\beta,\ep}^{-2}\pa_t\Phi_{\beta,\ep}\,d\mu 
\\
&=
2\int_{\Omega_n} \widetilde{w_n}(t)\Delta (\widetilde{w_n}\Phi_{\beta,\ep})\,dx
-
\int_{\Omega_n} |\widetilde{w_n}(t)|^2a\pa_t\Phi_{\beta,\ep}\,dx
\\
&=
-2\int_{\Omega_n} |\nabla \widetilde{w_n}(t)|^2\Phi_{\beta,\ep}\,dx
-
\int_{\Omega_n} |\widetilde{w_n}(t)|^2(a\pa_t\Phi_{\beta,\ep}-\Delta \Phi_{\beta,\ep})\,dx, 
\end{align*}
where we have put $\widetilde{w_n}(t)=w_n(t)\Phi_{\beta,\ep}^{-1}$. 
By using the property of $\Phi_{\beta,\ep}$ as a supersolution in Lemma \ref{lem:super-sol}, we deduce
\[
\int_{\Omega_n} |w_n(t)|^2\Phi_{\beta,\ep}(\cdot,t,1)^{-1}\,d\mu 
\leq 
\int_{\Omega_n} |f|^2\Phi_{\beta,\ep}(\cdot,0,1)^{-1}\,d\mu 
\]
and therefore we have
\begin{align}\label{eq:pre-west}
(1+t)^\beta
\int_{\Omega_n} |w_n(t)|^2\,d\mu 
\leq 
C\int_{\Omega_n} |f|^2\lr{x}^{(2-\alpha)\beta}\,d\mu 
\leq 
C\int_{\Omega} |f|^2\lr{x}^{(2-\alpha)\beta}\,d\mu. 
\end{align}
Then we prove $T_n(t)f\to T(t)f$ in $L^2_{d\mu}$ 
by employing Trotter--Kato convergence theorem 
(see Engel--Nagel \cite[Theorem 4.8]{EN00}).
For $f\in L^2_{d\mu}$, set $u_n=(1-L_n)^{-1}f$. 
Since the semigroups $\{T_n(t)\}_{t\geq 0}$ are positive, 
we may assume $f\geq 0$ and $u_n\geq 0$ without loss of generality. 
In view of maximum principle, the restriction $\chi_nu_n$
is nondecreasing with respect to $n$, 
where $\chi_n$ is the indicator function on $\Omega_n$. 
Moreover, observing that
\begin{align}\label{eq:weak-ell}
\int_{\Omega_n}
u_n \varphi
\,d\mu 
+
\int_{\Omega_n}
\nabla u_n \cdot\nabla \varphi
\,dx
=
\int_{\Omega_n}
f \varphi
\,d\mu, \quad \varphi\in H^1_0(\Omega_n) 
\end{align}
we see from the choice $\varphi=u_n$ that $\|u_n\|_{L^2_{d\mu}(\Omega_n)}\leq \|f\|_{L^2_{d\mu}}$ 
and $\|\nabla u_n\|_{L^2(\Omega_n)}\leq \|f\|_{L^2_{d\mu}}$. 
This yields that 
there exists $u\in L^2_{d\mu}\cap \dot{H}^1(\Omega)$ 
such that 
$\chi_{n}u_{n}\to u$ in $L^2_{d\mu}$ with $\|\nabla u\|_{L^2(\Omega)}\leq \|f\|_{L^2_{d\mu}}$. 
Since $u$ satisfies the Dirichlet boundary condition on $\pa\Omega$, $u$ 
belongs to $D(\mathfrak{a})$. Letting $n\to \infty$ in \eqref{eq:weak-ell}, 
we have
\[
\int_{\Omega}
u \varphi
\,d\mu 
+
\int_{\Omega}
\nabla u \cdot\nabla \varphi
\,dx
=
\int_{\Omega}
f \varphi
\,d\mu, \quad \varphi\in C_c^\infty(\Omega). 
\] 
Therefore $u\in D(L)$ and $u-Lu=f$. This gives $(I-L_{n})^{-1}f\to (I-L)^{-1}f$ in $L^2_{d\mu}$ as $n\to \infty$. 
The Trotter--Kato theorem implies $T_n(t)f\to T(t)f$ in $L^2_{d\mu}$ as $n\to \infty$. 
Consequently, \eqref{eq:pre-west} implies the desired estimate. 
The proof is complete. 
\end{proof}

\section{Application to weighted energy estimates for damped wave equations}\label{sec:damped}
In this section,
we discuss the asymptotic behavior of the solution to
the initial-boundary value problem of the damped wave equation
\begin{equation}%
\label{dw}
\begin{cases}
\pa_t^2u(x,t) - \Delta u(x,t) + a(x)\pa_tu(x,t)=0,
&
(x,t)\in \Omega\times (0,\infty),
\\
u(x,t)=0,&
(x,t)\in \pa\Omega\times (0,\infty),
\\
u(0,x)=u_0(x), 
\quad 
\pa_t u(0,x)=u_1(x), 
&
x \in \Omega.
\end{cases}
\end{equation}%
Here,
$\Omega$
is an exterior domain in
$\mathbb{R}^N$ with $N\ge 2$,
namely,
$\mathbb{R}^N \setminus \Omega$
is compact.
We assume that the boundary
$\partial\Omega$
is smooth.
We can also treat the case
$\Omega = \mathbb{R}^N$
with $N \ge 1$.
In that case, we omit the boundary condition from \eqref{dw}.

We assume that the coefficient of the damping term
$a(x)$
is a smooth positive function defined on
$\mathbb{R}^n$
and satisfying
\begin{align}%
\label{ass_a}
	\lim_{|x|\to \infty} (|x|^{\alpha} a(x)) = a_0
\end{align}%
with some
$\alpha \in [0,1)$
and
$a_0 > 0$.
The initial data satisfy
$(u_0, u_1) \in (H^2(\Omega)\cap H^1_0(\Omega))\times H^1_0(\Omega)$.

It is known that \eqref{dw} has a unique solution
\begin{align}%
\label{sol.class}
	u \in C^2([0,\infty); L^2(\Omega)) \cap C^1([0,\infty); H^1_0(\Omega)) \cap C([0,\infty); H^2(\Omega))
\end{align}%
(see \cite[Theorem 2]{Ikawa}).

In view of the validity of weighted Hardy inequality 
\[
\int_{\Omega}\lr{x}^{2(\sigma-1)}|u(x)|^2
\leq 
C
\int_{\Omega}\lr{x}^{2\sigma}|\nabla u(x)|^2
\,dx
\]
which crucially affects to  the validity of Lemma \ref{lem.hardy}, 
we will split the proofs of weighted energy estimates 
for multi-dimensional case and one-dimensional case.

\subsection{Weighted energy estimates for $N\ge 2$}
Let
$t_0 \ge 1$
be sufficiently large determined later and let
\begin{align}
\label{psi}
	\Psi(x,t; t_0) :=
		t_0 + t + A_{\varepsilon}(x),
\end{align}
where the function
$A_{\varepsilon}(x)$ is given in Lemma \ref{lem_A_ep}.
We first show the relation of $\Phi_{\beta,\ep}(x,t;t_0)$ and $\Psi(x,t;t_0)$.
\begin{lemma}\label{lem.phi.psi}
Let $\Phi_{\beta,\ep}(x,t;t_0)$
and
$\Psi(x,t;t_0)$ be defined in Definition \ref{phi.beta.ep} and \eqref{psi}, respectively.
Then the followings are hold.
\begin{itemize}
\item[(i)]
If $\beta \ge 0$, then
there exists a constant $C_{\alpha,\beta,\varepsilon} > 0$ such that
\begin{align*}
	|\Phi_{\beta,\varepsilon}(x,t;t_0) | \le C_{\alpha,\beta,\varepsilon} \Psi (x,t; t_0)^{-\beta}
\end{align*}
for any $(x,t) \in \mathbb{R}^N \times [0,\infty)$.
\item[(ii)]
If $\beta \in (0, \gamma_{\varepsilon})$,
then there exists a constant $c_{\alpha,\beta,\varepsilon} > 0$ such that
\begin{align*}
	|\Phi_{\beta,\varepsilon}(x,t;t_0) | \ge c_{\alpha,\beta,\varepsilon} \Psi (x,t; t_0)^{-\beta}
\end{align*}
for any $(x,t) \in \mathbb{R}^N \times [0,\infty)$.
\end{itemize}
\end{lemma}
This lemma directly follows from Lemma \ref{lem_varphi} and we omit the detail.

Next, we prepare a Hardy-type inequality with the weight function $\Psi$.
\begin{lemma}[Hardy-type inequality]\label{lem.hardy}
For every $w \in H^1_0(\Omega)$
having a compact support on $\mathbb{R}^N$
and
$\lambda > - \frac{N-2 + 2\varepsilon(N-\alpha)}{2-\alpha + \varepsilon(N-\alpha)}$,
there exists a positive constant
$C = C_{N,\alpha,\varepsilon,\lambda}$
such that
\begin{align*}
	\int_{\Omega}|w|^2 a(x) \Psi^{\lambda-1} \,dx
		&\le C \int_{\Omega} |\nabla w|^2 \Psi^{\lambda}\,dx.
\end{align*}
\end{lemma}
\begin{remark}
(i) The constant $C_{N,\alpha,\varepsilon,\lambda}$
in the above lemma is explicitly given by
\begin{align*}
	C_{N,\alpha,\varepsilon,\lambda}=
		4 \left(\frac{2-\alpha}{N-\alpha}+\varepsilon\right)
		\min\left\{ 1-\varepsilon, 1-\varepsilon
		+ (\lambda-1)\left( \frac{2-\alpha}{N-\alpha} + \varepsilon \right) \right\}^{-2}.
\end{align*}
(ii)
Lemma \ref{lem.hardy} holds even when $\Omega = \mathbb{R}^N$ with $N=1$.
However, due to the restriction on
$\lambda$,
we cannot apply it to weighted energy estimates.
This is the difference with one-dimensional case.
\end{remark}
\begin{proof}[Proof of Lemma \ref{lem.hardy}]
The proof is similar to that of \cite[Lemma 3.5]{SoWa19_CCM}.
First, noting $\nabla \Psi = \nabla A_\ep$ and $\Delta \Psi = \Delta A_\ep$,
and using Lemma \ref{lem_A_ep},
we calculate
\begin{align*}
	{\rm div\,}\left( \Psi^{\lambda-1} \nabla \Psi \right)
	&= (\Delta \Psi) \Psi^{\lambda-1} + (\lambda-1) |\nabla \Psi|^2 \Psi^{\lambda-2} \\
	&\ge (1-\varepsilon) a(x) \Psi^{\lambda-1}
		+ (\lambda-1) |\nabla A_\ep(x)|^2 \Psi^{\lambda-2} \\
	&= \left[ (1-\varepsilon) (t_0 +t + A_{\varepsilon}(x) )
		+ (\lambda-1) \frac{|\nabla A_\ep(x)|^2}{a(x)} \right] a(x) \Psi^{\lambda-2} \\
	&= \left[ (1-\varepsilon)(t_0 +t)
		+ \left( 1-\varepsilon + (\lambda-1)\frac{|\nabla A_\ep(x)|^2}{a(x)A_\ep(x)} \right)
			A_\ep(x) \right]
			a(x) \Psi^{\lambda-2} \\
	&\ge
		\min\left\{ 1-\varepsilon, 1-\varepsilon
				+ (\lambda-1)\left( \frac{2-\alpha}{N-\alpha} + \varepsilon \right) \right\}
				a(x) \Psi^{\lambda-1}.
\end{align*}
Since $\varepsilon \in (0,1)$ and
$\lambda > - \frac{N-2 + 2\varepsilon(N-\alpha)}{2-\alpha + \varepsilon(N-\alpha)}$,
all members in the minimum are positive.
On the one hand, the integration by parts and the Schwarz inequality lead to
\begin{align*}
	\int_{\Omega} |w|^2 {\rm div\,} \left( \Psi^{\lambda-1} \nabla \Psi \right) \,dx
	&= - 2 \int_{\Omega} w (\nabla w \cdot \nabla \Psi ) \Psi^{\lambda-1} \,dx \\
	&\le 2 \left( \int_{\Omega} |w|^2 a(x) \Psi^{\lambda-1} \,dx \right)^{1/2}
		\left( \int_{\Omega} |\nabla w|^2 \frac{|\nabla \Psi|^2}{a(x)} \Psi^{\lambda-1} \,dx \right)^{1/2} \\
	&\le 2 \left( \int_{\Omega} |w|^2 a(x) \Psi^{\lambda-1} \,dx \right)^{1/2}
		\left( \left(\frac{2-\alpha}{N-\alpha} + \varepsilon \right)
			\int_{\Omega} |\nabla w|^2 \Psi^{\lambda} \,dx \right)^{1/2}.
\end{align*}
Here we have used that
\[
	\frac{|\nabla \Psi|^2}{a(x)} = \frac{|\nabla A(x)|^2}{a(x)A(x)} A(x)
		\le \left(\frac{2-\alpha}{N-\alpha} + \varepsilon \right) \Psi(x,t;t_0)
\]
holds by Lemma \ref{lem_A_ep}.
Putting all the estimates together, we obtain the desired assertion.
\end{proof}

\begin{lemma}[\cite{So18a} Lemma 2.5]\label{lem.deltaphi}
For $\Phi \in C^2(\overline{\Omega})$, $u\in H^2(\Omega) \cap H^1_0(\Omega)$ and $\delta \in (0, 1/2)$,
we have
\begin{align*}
	\int_{\Omega} u \Delta u \Phi^{-1+2\delta}\,dx
		&\le - \frac{\delta}{1-\delta} \int_{\Omega} |\nabla u|^2 \Phi^{-1+2\delta}\,dx
			+ \frac{1-2\delta}{2} \int_{\Omega} u^2 (\Delta \Phi) \Phi^{-2+2\delta} \,dx.
\end{align*}
\end{lemma}
The proof in done by integration by parts and
can be found in \cite[Lemma 2.5]{So18a} and we omit the detail.

\begin{definition}[Weighted energy]\label{def.en}
Let
$\delta \in (0,1/2)$,
$\varepsilon \in (0,1)$,
$\lambda \in (0, (1-2\delta)\gamma_{\varepsilon})$
and
$\beta = \lambda/(1-2\delta)$.
Let
$t_0 \ge 1$
and
$\nu > 0$
be sufficiently large and small, respectively, and determined later.
We define the weighted energy
\begin{align}
\label{e1}
	E_{1}(t; t_0)
		&:= \int_{\Omega}
			\left( |\nabla u(x,t)|^2 + |\partial_t u(x,t)|^2 \right)
			\Psi(x,t;t_0)^{\lambda+\frac{\alpha}{2-\alpha}} \,dx,\\
\label{e0}
	E_{0}(t;t_0)
		&:= \int_{\Omega} \left( 2u(x,t) \partial_t u(x,t) + a(x) |u(x,t)|^2 \right)
			\Phi_{\beta,\varepsilon}(x,t;t_0)^{-1+2\delta} \,dx,\\
\label{enu}
	E(t;t_0,\nu)
		&:= E_{1}(t; t_0) + \nu E_0(t;t_0),\\
\label{e1til}
	\widetilde{E}_{1}(t; t_0)
		&:= (t_0+t) \int_{\Omega}
			\left( |\nabla u(x,t)|^2 + |\partial_t u(x,t)|^2 \right)
			\Psi(x,t;t_0)^{\lambda} \,dx
\end{align}
for $t \ge 0$.
\end{definition}
Our strategy of the weighted energy estimates is the following:
First, combining the estimates for
$E_{1}(t; t_0)$ and $E_{0}(t;t_0)$, we give an energy estimate for $E(t;t_0,\nu)$.
Then, using it, we derive the boundedness of
$\widetilde{E}_{1}(t; t_0)$,
which gives a sharper decay estimate for $(\nabla u, \partial_t u)$.
The main result of this subsection is the following:
\begin{theorem}\label{thm.e1}
Assume \eqref{ass_a}. Then 
there exist
$t_{*}\ge 1$ and $\nu > 0$
such that
for any $t_0 \ge t_{*}$
the following holds:
suppose that the initial data satisfy
\begin{align*}
	I_0 := \int_{\Omega} (|\nabla u_0(x)|^2 + |u_1(x)|^2) \Psi(x,0;t_0)^{\lambda+\frac{\alpha}{2-\alpha}} \,dx
		+ \int_{\Omega} a(x) |u_0(x)|^2  \Psi(x,0;t_0)^{\lambda} \,dx
		< \infty.
\end{align*}
Let
$u$
be the solution of \eqref{dw} in the class \eqref{sol.class}.
Then, we have
\begin{align*}
	&E_1(t;t_0) + \widetilde{E}_{1}(t; t_0) + \int_{\Omega} a(x) |u(x,t)|^2 \Psi(x,t;t_0)^{\lambda} \,dx \\
	&\quad + \int_0^t \int_{\Omega} |\nabla u(x,\tau)|^2 \Psi(x,\tau;t_0)^{\lambda} \,dx d\tau \\
	&\quad
		+ \int_0^t
			\int_{\Omega} a(x) |\partial_t u(x,\tau)|^2 \Psi(x,\tau;t_0)^{\lambda}
			\left[ (t_0+ \tau) + \Psi(x,\tau;t_0)^{\frac{\alpha}{2-\alpha}} \right]
			\,dx d\tau \\
	&\le C I_0
\end{align*}
for
$t \ge 0$
with some constant
$C = C(N,\alpha,\delta,\varepsilon,\lambda,t_0,\nu)>0$.
\end{theorem}

\subsubsection{Proof of Theorem \ref{thm.e1}}

In the proof of weighted energy estimates, 
we will assume that the initial data $(u_0,u_1)$ (and also the solution $u$ by finite propagation property) 
are compactly supported. 
All estimates proved below can be extended to 
the case of non-compactly supported initial data 
via an approximation with a cut-off procedure.

We split the proof of Theorem \ref{thm.e1} into the following four lemmas.
\begin{lemma}\label{lem.e1}
Under the assumption on Theorem \ref{thm.e1}, there exists a constant
$t_1 \ge 1$
such that
for any
$t_0 \ge t_1$ and $t \ge 0$, we have
\begin{align*}%
	\frac{d}{dt} E_1(t;t_0)
		\le - \int_{\Omega} a(x) |\partial_t u(x,t)|^2 \Psi(x,t;t_0)^{\lambda+\frac{\alpha}{2-\alpha}}\,dx
			+ C \int_{\Omega} |\nabla u(x,t)|^2 \Psi(x,t;t_0)^{\lambda+\frac{\alpha}{2-\alpha}-1}\,dx
\end{align*}%
with some constant
$C=C(N,\alpha,\varepsilon,\lambda,t_0) >0$.
\end{lemma}
\begin{proof}
Since $u$ is a solution of \eqref{dw}, we compute
\begin{align*}%
	\frac{d}{dt} E_1(t;t_0)
	&= 2 \int_{\Omega} (\nabla \partial_t u \cdot \nabla u + \partial_t u \partial_t^2 u )
			\Psi^{\lambda+\frac{\alpha}{2-\alpha}} \,dx \\
	&\quad + \left( \lambda + \frac{\alpha}{2-\alpha} \right)
		\int_{\Omega} ( |\nabla u|^2 + |\partial_t u|^2 ) \Psi^{\lambda+\frac{\alpha}{2-\alpha}-1}\,dx \\
	&= -2 \int_{\Omega} a(x) |\partial_t u|^2 \Psi^{\lambda+\frac{\alpha}{2-\alpha}} \,dx
		-2 \left( \lambda + \frac{\alpha}{2-\alpha} \right)
			\int_{\Omega} \partial_t u (\nabla u \cdot \nabla \Psi) \Psi^{\lambda+\frac{\alpha}{2-\alpha}-1} \,dx \\
	&\quad + \left( \lambda + \frac{\alpha}{2-\alpha} \right)
		\int_{\Omega} ( |\nabla u|^2 + |\partial_t u|^2 ) \Psi^{\lambda+\frac{\alpha}{2-\alpha}-1}\,dx.
\end{align*}%
By the Schwarz inequality
\begin{align*}%
	\left| -2 \left( \lambda + \frac{\alpha}{2-\alpha} \right)
		\partial_t u (\nabla u \cdot \nabla \Psi) \right|
		&\le \frac{a(x)}{2} |\partial_t u|^2 \Psi + C |\nabla u|^2 \frac{|\nabla \Psi|^2}{a(x)\Psi}
\end{align*}%
and noting
\begin{align}%
\label{psi.ratio}
	\frac{|\nabla \Psi|^2}{a(x)\Psi} \le \frac{|\nabla A_{\varepsilon} (x)|^2}{a(x) A_{\varepsilon}(x)}
		\le \frac{2-\alpha}{N-\alpha} + \varepsilon,
\end{align}%
which follows from \eqref{A3},
we conclude
\begin{align*}%
	\frac{d}{dt} E_1(t;t_0)
		&\le \int_{\Omega} |\partial_tu|^2
			\left( -2a(x) + \frac{a(x)}{2} + \left(\lambda + \frac{\alpha}{2-\alpha} \right) \Psi^{-1} \right)
			\Psi^{\lambda+\frac{\alpha}{2-\alpha}} \,dx
		+ C \int_{\Omega} |\nabla u|^2 \Psi^{\lambda + \frac{\alpha}{2-\alpha} -1} \,dx.
\end{align*}%
Finally, by
$\Psi^{-1} \le t_0^{-1+\frac{\alpha}{2-\alpha}} A(x)^{-\frac{\alpha}{2-\alpha}} \le C t_0^{-1+\frac{\alpha}{2-\alpha}} a(x)$,
taking $t_1 \ge 1$ sufficiently large, we have the desired estimate.
\end{proof}

\begin{lemma}\label{lem.e0}
Under the assumption on Theorem \ref{thm.e1}, there exists a constant $t_2 \ge 1$ such that
for any
$t_0 \ge t_2$ and $t \ge 0$,
we have
\begin{align*}
	\frac{d}{dt} E_{0}(t;t_0)
		&\le
		-  \eta_0 \int_{\Omega} |\nabla u(x,t) |^2 \Psi(x,t;t_0)^{\lambda}\,dx
		+ C \int_{\Omega} |\partial_t u(x,t)|^2 \Psi(x,t;t_0)^{\lambda} \,dx
\end{align*}
with some constants
$\eta_0 = \eta_0 (N,\alpha,\varepsilon,\delta,\lambda,t_0) > 0$
and
$C = C(N,\alpha,\varepsilon,\delta,\lambda,t_0) > 0$.
\end{lemma}
\begin{proof}
Since $u$ is a solution of \eqref{dw}, we compute
\begin{align*}
	\frac{d}{dt} E_{0}(t;t_0)
		&= 2 \int_{\Omega} |\partial_t u|^2 \Phi_{\beta,\varepsilon}^{-1+2\delta} \,dx
			+ 2 \int_{\Omega} u( \partial_t^2 u + a(x) \partial_t u ) \Phi_{\beta,\varepsilon}^{-1+2\delta}\,dx \\
		&\quad
			-(1-2\delta) \int_{\Omega} (2 u\partial_t u + a(x) |u|^2 )
					\Phi_{\beta,\varepsilon}^{-2+2\delta} \partial_t \Phi_{\beta,\varepsilon}  \,dx \\
		&= 2 \int_{\Omega} |\partial_t u|^2 \Phi_{\beta,\varepsilon}^{-1+2\delta} \,dx
			+ 2 \int_{\Omega} u \Delta u \Phi_{\beta,\varepsilon}^{-1+2\delta}\,dx \\
		&\quad -2(1-2\delta) \int_{\Omega} u \partial_t u 
				\Phi_{\beta,\varepsilon}^{-2+2\delta} \partial_t \Phi_{\beta,\varepsilon}  \,dx
			-(1-2\delta) \int_{\Omega} a(x) |u|^2
				\Phi_{\beta,\varepsilon}^{-2+2\delta} \partial_t \Phi_{\beta,\varepsilon}  \,dx.
\end{align*}
Lemma \ref{lem.deltaphi} with $\Phi = \Phi_{\beta,\varepsilon}$ yields
\begin{align*}
	\frac{d}{dt} E_{0}(t;t_0)
		&\le 2 \int_{\Omega} |\partial_t u|^2 \Phi_{\beta,\varepsilon}^{-1+2\delta} \,dx
			- \frac{2 \delta}{1-\delta} \int_{\Omega} |\nabla u|^2 \Phi_{\beta,\varepsilon}^{-1+2\delta}\,dx \\
		&\quad -2(1-2\delta) \int_{\Omega} u \partial_t u 
				\Phi_{\beta,\varepsilon}^{-2+2\delta} \partial_t \Phi_{\beta,\varepsilon} \,dx
			- (1-2\delta) \int_{\Omega} |u|^2
				\Phi_{\beta,\varepsilon}^{-2+2\delta}
					(a(x)\partial_t \Phi_{\beta,\varepsilon} - \Delta \Phi_{\beta,\varepsilon}) \,dx.
\end{align*}
By Lemma \ref{lem:super-sol}, $\Phi_{\beta,\varepsilon}$ satisfies
\begin{align*}
	a(x)\partial_t \Phi_{\beta,\varepsilon} - \Delta \Phi_{\beta,\varepsilon} \ge 0.
\end{align*}
Moreover, by noting
$\beta \in (0,\gamma_{\varepsilon})$,
we apply Lemmas \ref{lem.phi.psi} and \ref{lem.dtphi} to obtain
\begin{align}
\nonumber
	\frac{d}{dt} E_{0}(t;t_0)
		&\le \frac{2}{c_{\alpha,\beta,\varepsilon}}
			\int_{\Omega} |\partial_t u|^2 \Psi^{\lambda} \,dx
		 - \frac{2 \delta}{(1-\delta)C_{\alpha,\beta,\varepsilon}}
		 	\int_{\Omega} |\nabla u|^2 \Psi^{\lambda}\,dx \\
\label{eq.de0}
		&\quad
			+ \frac{2(1-2\delta)\beta C_{\alpha,\beta+1,\varepsilon}}{c_{\alpha,\beta,\varepsilon}^{2-2\delta}}
			 \int_{\Omega} |u| |\partial_t u| \Psi^{\lambda-1} \,dx
\end{align}
With the aid of Lemma \ref{lem.hardy}, the last term is estimated as
\begin{align*}
	 \int_{\Omega} |u| |\partial_t u| \Psi^{\lambda-1} \,dx
	 &\le \left(  \int_{\Omega} a(x) |u|^2 \Psi^{\lambda-1} \,dx \right)^{1/2}
	 	\left(  \int_{\Omega} a(x)^{-1} |\partial_t u|^2 \Psi^{\lambda-1} \,dx \right)^{1/2} \\
	&\le C \left(  \int_{\Omega} |\nabla u|^2 \Psi^{\lambda} \,dx \right)^{1/2}
		\left(  \int_{\Omega} a(x)^{-1} |\partial_t u|^2 \Psi^{\lambda-1} \,dx \right)^{1/2} \\
	&\le C(t+t_0)^{-\frac{1-\alpha}{2-\alpha}}
		\left(  \int_{\Omega} |\nabla u|^2 \Psi^{\lambda} \,dx \right)^{1/2}
		\left(  \int_{\Omega} |\partial_t u|^2 \Psi^{\lambda} \,dx \right)^{1/2} \\
	&\le C \int_{\Omega} |\partial_t u|^2 \Psi^{\lambda} \,dx 
		+ C t_0^{-\frac{2(1-\alpha)}{2-\alpha}} \int_{\Omega} |\nabla u|^2 \Psi^{\lambda} \,dx.
\end{align*}
Here, for the third inequality step we have used the following:
\begin{align*}
	\frac{1}{a(x) \Psi(x,t;t_0)}
		&\le C \frac{\langle x \rangle^{\alpha}}{t+t_0 + A_\ep(x)}
		\le C \frac{1}{(t + t_0)^{\frac{2(1-\alpha)}{2-\alpha}}}
			\frac{\langle x \rangle^{\alpha}}{A_\ep(x)^{\frac{\alpha}{2-\alpha}}}
		\le C \frac{1}{(t + t_0)^{\frac{2(1-\alpha)}{2-\alpha}}}.
\end{align*}
Thus, taking
$t_2 \ge 1$
sufficiently large,
we conclude
\begin{align*}
	\frac{d}{dt} E_{0}(t;t_0)
		&\le - \eta_0 \int_{\Omega} |\nabla u|^2 \Psi^{\lambda}\,dx
			+ C \int_{\Omega} |\partial_t u|^2 \Psi^{\lambda} \,dx
\end{align*}
for $t_0 \ge t_2$ and $t \ge 0$ with some constant $\eta_0>0$.
\end{proof}

Noting $\Psi^{\lambda} \le C a(x) \Psi^{\lambda + \frac{\alpha}{2-\alpha}}$
and
combining Lemmas \ref{lem.e1} and \ref{lem.e0},
we have an energy estimate for $E(t;t_0,\nu)$.
\begin{lemma}\label{lem.e}
Under the assumption on \ref{thm.e1}, for any
$t_0 \ge \max\{ t_1, t_2 \}$,
there exists
$\nu = \nu(N,\alpha,\varepsilon,\delta,\lambda,t_0) > 0$
such that
\begin{align*}
	E(t;t_0,\nu) &\ge
		c \int_{\Omega} (|\nabla u(x,t)|^2 + |\partial_t u(x,t)|^2) \Psi(x,t;t_0)^{\lambda+\frac{\alpha}{2-\alpha}} \,dx
			+ c \int_{\Omega} a(x) |u(x,t)|^2 \Psi(x,t;t_0)^{\lambda} \,dx
\end{align*}
and
\begin{align*}
	&E(t;t_0,\nu)
		+ \int_0^t \int_{\Omega}
			(|\nabla u(x,\tau)|^2 + |\partial_t u(x,\tau)|^2) \Psi(x,\tau;t_0)^{\lambda} \,dxd\tau \\
	&\quad + \int_0^t \int_{\Omega}
		a(x) |\partial_t u(x,\tau)|^2 \Psi(x,\tau;t_0)^{\lambda+\frac{\alpha}{2-\alpha}}\,dxd\tau \\
	&\le C E(0;t_0,\nu)
\end{align*}
hold for
$t\ge 0$
with some constants
$c = c(N,\alpha,\varepsilon,\delta,\lambda,t_0) > 0$
and
$C = C(N,\alpha,\varepsilon,\delta,\lambda,t_0) > 0$.
\end{lemma}

Finally, using Lemma \ref{lem.e},
we give the following energy estimate for $\widetilde{E}_1(t;t_0)$.
\begin{lemma}\label{lem.e1.til}
Under the assumption on Theorem \ref{thm.e1}, for any
$t_0 \ge \max\{ t_1, t_2 \}$,
there exists
$\nu = \nu(N,\alpha,\varepsilon,\delta,\lambda,t_0) > 0$
such that
\begin{align*}
	\widetilde{E}_{1}(t;t_0)
		+ \int_0^t (t_0+\tau) \int_{\Omega}
			a(x) |\partial_t u(x,\tau)|^2 \Psi(x,\tau;t_0)^{\lambda} \,dxd\tau
	\le  C E(0;t_0)
\end{align*}
with some constant
$C = C(N,\alpha,\varepsilon,\delta,\lambda,t_0) > 0$.
\end{lemma}
\begin{proof}
By integration by parts and the Schwarz inequality, we compute
\begin{align*}
	\frac{d}{dt} \widetilde{E}_{1}(t;t_0)
		&= \int_{\Omega} \left( |\nabla u|^2 + |\partial_t u|^2 \right)
			[ \Psi^{\lambda} + \lambda (t_0+t) \Psi^{\lambda-1} ] \,dx \\
		&\quad + 2(t_0+t) \int_{\Omega}
				\left( \nabla \partial_t u \cdot \nabla u + \partial_t u \partial_t^2 u \right)
				\Psi^{\lambda} \,dx \\
		&\le (\lambda +1) \int_{\Omega}
			\left( |\nabla u|^2 + |\partial_t u|^2 \right) \Psi^{\lambda} \,dx \\
		&\quad
			-2(t_0+t) \int_{\Omega} a(x) |\partial_t u|^2 \Psi^{\lambda} \,dx
			-2\lambda (t_0+t) \int_{\Omega}
				\partial_t u (\nabla u \cdot \nabla \Psi) \Psi^{\lambda -1} \,dx \\
		&\le (\lambda +1) \int_{\Omega}
			\left( |\nabla u|^2 + |\partial_t u|^2 \right) \Psi^{\lambda} \,dx \\
		&\quad -(t_0+t) \int_{\Omega} a(x) |\partial_t u|^2 \Psi^{\lambda} \,dx
			+ \lambda^2 (t_0+t) \int_{\Omega}
					|\nabla u|^2 \frac{|\nabla \Psi|^2}{a(x) \Psi} \Psi^{\lambda-1} \,dx \\
		&\le C \int_{\Omega}
			\left( |\nabla u|^2 + |\partial_t u|^2 \right) \Psi^{\lambda} \,dx
			- (t_0+t) \int_{\Omega} a(x) |\partial_t u|^2 \Psi^{\lambda} \,dx.
\end{align*}
Integrating the above on $[0,t]$, applying Lemma \ref{lem.e}, and noting
$  \widetilde{E}_{1}(0;t_0) \le C E(0;t_0 ,\nu)$,
we deduce
\begin{align*}
	\widetilde{E}_{1}(t;t_0)
		+ \int_0^t (t_0+\tau) \int_{\Omega} a(x) |\partial_t u|^2 \Psi^{\lambda} \,dx
		\le C E(0;t_0, \nu),
\end{align*}
which completes the proof.
\end{proof}

Theorem \ref{thm.e1} immediately follows from
Lemmas \ref{lem.e1} and \ref{lem.e1.til}.

\subsection{Weighted energy estimates for $N=1$}
In the one-dimensional case, instead of Lemma \ref{lem.hardy},
we use a modified weight function
\begin{align*}
	\widetilde{\Phi}_{\beta,\varepsilon}(x,t;t_0) :=
		\left( 2- \frac{1}{(t_0+t)^{\frac{2(1-\alpha)}{2-\alpha}}}  \right) \Phi_{\beta,\varepsilon}(x,t;t_0).
\end{align*}
Then, $\widetilde{\Phi}_{\beta, \varepsilon}$ satisfies
\begin{align*}
	a(x) \partial_t \widetilde{\Phi}_{\beta,\varepsilon}(x,t;t_0) 
		- \Delta \widetilde{\Phi}_{\beta,\varepsilon}(x,t;t_0)
		\ge \frac{1-\alpha}{2-\alpha} a(x)
				(t_0+t)^{-2 + \frac{\alpha}{2-\alpha}} \widetilde{\Phi}_{\beta,\varepsilon}(x,t;t_0).
\end{align*}
Using
$\widetilde{\Phi}_{\beta, \varepsilon}$,
we modify the definition of $E_0(t;t_0)$ as
\begin{align*}
	E_0(t;t_0) = E_{0}(t;t_0)
		&:= \int_{\Omega} \left( 2u(x,t) \partial_t u(x,t) + a(x) |u(x,t)|^2 \right)
			\widetilde{\Phi}_{\beta,\varepsilon}(x,t;t_0)^{-1+2\delta} \,dx.
\end{align*}
Therefore, in the proof of Lemma \ref{lem.e0},
instead of \eqref{eq.de0}, we obtain
\begin{align*}
	\frac{d}{dt} E_{0}(t;t_0)
		&\le C \int_{\mathbb{R}} |\partial_t u|^2 \Psi^{\lambda} \,dx
			 - \eta \int_{\mathbb{R}} |\nabla u|^2 \Psi^{\lambda}\,dx
			 - \eta (t_0+t)^{-2 + \frac{\alpha}{2-\alpha}} \int_{\mathbb{R}} a(x)|u|^2 \Psi^{\lambda}\,dx \\
		&\quad + C \int_{\mathbb{R}} |u| |\partial_t u| \Psi^{\lambda-1} \,dx
\end{align*}
with some
$C, \eta>0$.
The last term is estimated as
\begin{align*}
	\int_{\mathbb{R}} |u| |\partial_t u| \Psi^{\lambda-1} \,dx
	&\le \tilde{\eta} \int_{\mathbb{R}} |u|^2 \Psi^{\lambda-2} \,dx
		+ C \int_{\mathbb{R}} |\partial_t u|^2 \Psi^{\lambda}\,dx.
\end{align*}
Noting
$\Psi^{-2} \le (t_0+t)^{-2+\frac{\alpha}{2-\alpha}} A_\ep(x)^{-\frac{\alpha}{2-\alpha}}
\le C (t_0+t)^{-2+\frac{\alpha}{2-\alpha}} a(x)$
and taking $\tilde{\eta}$ sufficiently small, we have the same conclusion of Lemma \ref{lem.e0}.
The rest part is completely the same as in the case $N\ge 2$,
and we have the same conclusion of Theorem \ref{thm.e1} in the case $N=1$.

\subsection{Weighted energy estimates for higher order derivatives}
In this subsection, we discuss weighted energy estimates for
higher order derivatives of the solution.
For $k \in \mathbb{N}$,
We say that the initial data satisfy the compatibility condition of order $k$ if
\begin{align*}
	u_{\ell} = \Delta u_{\ell-2} - a(x) u_{\ell-1},\quad
	(u_{\ell-1}, u_{\ell})\in (H^2\cap H^1_0(\Omega)) \times H^1_0(\Omega),\quad
	\quad (\ell=2,\ldots,k+1)
\end{align*}
can be successively defined.
For $k \in \mathbb{N}$,
It is known that if
$(u_0,u_1) \in (H^{k+2} \cap H^1_0(\Omega)) \times (H^{k+1} \cap H^1_0(\Omega))$
fulfill the compatibility condition of order
$k$,
then the solution of \eqref{dw} satisfies
\begin{align}
\label{sol.class.k}
	u \in \bigcap_{\ell=0}^{k+2} C^{\ell} ([0,\infty); H^{k-\ell+2}(\Omega))
\end{align}
in addition to \eqref{sol.class}
(see \cite[Theorem 2]{Ikawa}).

\begin{definition}[Weighted energy of higher order derivatives]\label{def.ek}
Let
$k \in \mathbb{N}$,
$\delta \in (0,1/2)$,
$\varepsilon \in (0,1)$,
and
$\lambda \in (0, (1-2\delta)\gamma_{\varepsilon})$.
Let
$t_0 \ge 1$
and
$\nu_{k,j} > 0$
with $j=0,1,\ldots,2k$
be sufficiently large and small, respectively, and determined later.
For a function
$w = w(x,t)$,
we define the weighted energy for $t \ge 0$ by
\begin{align}
\label{e1k}
	E_{1}^{(k,j)}[w] (t; t_0)
		&:= (t_0+t)^j \int_{\Omega}
			\left( |\nabla w(x,t)|^2 + | \partial_t w(x,t)|^2 \right)
			\Psi(x,t;t_0)^{\lambda+ (2k+1-j)\frac{\alpha}{2-\alpha}} \,dx
\end{align}
for $j=0,1,\ldots,2k+1$, and
\begin{align}
\label{e0k}
	E_{0}^{(k,j)} [w](t;t_0)
		&:= (t_0+t)^j \int_{\Omega} \left( 2w(x,t) \partial_t w(x,t) + a(x) |w(x,t)|^2 \right)
			\Psi (x,t;t_0)^{\lambda + (2k-j)\frac{\alpha}{2-\alpha}} \,dx,\\
\label{enuk}
	E^{(k,j)} [w] (t;t_0,\nu_{k,j})
		&:= E_{1}^{(k,j)}(t; t_0) + \nu_{k,j} E_0^{(k,j)} (t;t_0)
\end{align}
for $j=0,1,\ldots,2k$.
\end{definition}

The main result of this subsection is the following weighted energy estimates
for time derivatives of the solution,
which improves our previous result in \cite[Theorem 4.1]{SoWa19_CCM}.

\begin{theorem}\label{thm.ek}
Let $k\in \mathbb{N}$
and let the initial data $(u_0,u_1)$ satisfy the compatibility condition of order $k$.
Then, there exist
$t_{*} \ge 1$
and
$\nu_{k,j} > 0$
with $j=0,1,\ldots,2k$
such that
for any $t_0 \ge t_{*}$
the following holds:
Assume that the initial data satisfy
\begin{align*}
	I_{k} := \sum_{\ell=0}^k
		\left[ \int_{\Omega} (|\nabla u_{\ell}(x)|^2 +| u_{\ell+1} (x)|^2 )
		\Psi(x,0;t_0)^{\lambda + (2\ell+1) \frac{\alpha}{2-\alpha}} \,dx
		+ \int_{\Omega} a(x) |u_{\ell}(x)|^2
			\Psi(x,0;t_0)^{\lambda + 2\ell \frac{\alpha}{2-\alpha}} \,dx \right] < \infty.
\end{align*}
Let $u$ be the corresponding solution in the class \eqref{sol.class} and \eqref{sol.class.k}.
Then, we have
\begin{align*}
	&\sum_{j=0}^{2k+1} E_1^{(k,j)} [\partial_t^k u](t;t_0)
		+ \sum_{j=0}^{2k} (t_0+t)^j \int_{\Omega} a(x) |\partial_t^k u(x,t)|^2
				\Psi(x,t;t_0)^{\lambda + (2k-j)\frac{\alpha}{2-\alpha}} \,dx \\
	&\quad + \sum_{j=0}^{2k+1} \int_0^t (t_0+\tau)^{j}
			\int_{\Omega} a(x) |\partial_t^{k+1} u(x,\tau)|^2
				\Psi(x,\tau;t_0)^{\lambda + (2k+1-j)\frac{\alpha}{2-\alpha}} \,dxd\tau \\
	&\quad + \sum_{j=0}^{2k} \int_0^t (t_0+\tau)^{j}
			\int_{\Omega} |\nabla \partial_t^k u(x,\tau)|^2
				\Psi(x,\tau;t_0)^{\lambda + (2k-j)\frac{\alpha}{2-\alpha}} \,dxd\tau \\
	&\le C I_{k}
\end{align*}
for
$t \ge 0$
with some constant
$C = C(k, N,\alpha, \delta, \varepsilon,\lambda,t_0,\nu_{k,0},\ldots, \nu_{k,2k})>0$.
\end{theorem}
\begin{remark}\label{rem.ek}
If we formally take $k=0$ in the above theorem,
then we have the same conclusion of Theorem \ref{thm.e1}.
In this sense we interpret that the above theorem is also valid for $k=0$.
\end{remark}

\subsubsection{Proof of Theorem \ref{thm.ek}}
We prove Theorem \ref{thm.ek} by induction.
The case $k=0$ has already done by Theorem \ref{thm.e1} (see Remark \ref{rem.ek}).
We assume that  Theorem \ref{thm.ek} is valid for $k-1$.

Next, for the induction step,
we prove the following lemma,
which shows that if a solution of the damped wave equation \eqref{dw}
has a certain space-time bound, then it decays faster than general cases.
\begin{lemma}\label{lem.e.w}
Let $k \in \mathbb{N}$.
Let $(w_0, w_1)$ satisfy the compatibility condition of order $1$
and $w$ be the corresponding solution of \eqref{dw} with the initial data $(w_0, w_1)$.
Then, there exists
$t_* \ge 1$ and $\nu_{k,j}>0$ with $j=0,1,\ldots,2k$ such that for any $t_0 \ge t_*$,
the following holds:
Assume that the initial data satisfy
\begin{align*}
	I = \int_{\Omega} (|\nabla w_0(x)|^2 + |w_1(x)|^2) \Psi(x,0; t_0)^{\lambda+(2k+1)\frac{\alpha}{2-\alpha}} \,dx
		+ \int_{\Omega} a(x) |w_0(x)|^2 \Psi(x,0;t_0)^{\lambda+2k\frac{\alpha}{2-\alpha}} \,dx < \infty
\end{align*} 
and the solution $w$ satisfies
\begin{align}
\label{induction.assumption}
	K=\sum_{j=0}^{2(k-1)+1} \int_0^{\infty} (t_0+\tau)^{j}
		\int_{\Omega} a(x) |w(x,\tau)|^2 \Psi(x,\tau;t_0)^{\lambda+(2k-1+j)\frac{\alpha}{2-\alpha}} \,dxd\tau
	<\infty.
\end{align}
Then, we have
\begin{align}
\nonumber
	&\sum_{j=0}^{2k+1} E_1^{(k,j)} [w](t;t_0)
		+ \sum_{j=0}^{2k} (t_0+t)^j \int_{\Omega} a(x) |w(x,t)|^2
				\Psi(x,t;t_0)^{\lambda + (2k-j)\frac{\alpha}{2-\alpha}} \,dx \\
\nonumber
	&\quad + \sum_{j=0}^{2k+1} \int_0^t (t_0+\tau)^{j}
			\int_{\Omega} a(x) |\partial_t w(x,\tau)|^2
				\Psi(x,\tau;t_0)^{\lambda + (2k+1-j)\frac{\alpha}{2-\alpha}} \,dxd\tau \\
\nonumber
	&\quad + \sum_{j=0}^{2k} \int_0^t (t_0+\tau)^{j}
			\int_{\Omega} |\nabla w(x,\tau)|^2
				\Psi(x,\tau;t_0)^{\lambda + (2k-j)\frac{\alpha}{2-\alpha}} \,dxd\tau \\
\label{eq.e.w}
	&\le C (I+K)
\end{align}
for
$t \ge 0$
with some constant
$C = C(k, N,\alpha,\delta,\varepsilon,\lambda,t_0,\nu_{k,0},\ldots, \nu_{k,2k})>0$.
\end{lemma}
Since the assumption of induction ensures the condition \eqref{induction.assumption}
when $w = \partial_t^k u$,
we obtain the induction step of the proof of Theorem \ref{thm.ek} from Lemma \ref{lem.e.w}.
Therefore, it suffices to show Lemma \ref{lem.e.w}.

The proof of Lemma \ref{lem.e.w} is highly technical.
However, the principle is simple, that is, the assumptions of the space-time bound \eqref{induction.assumption}
and the bound of a certain weighted energy of initial data
produce faster energy decay estimates of solutions.
Actually, in the first step, by using \eqref{induction.assumption}, we give an estimate of
$E^{(k,0)}[w](t;t_0)$.
As a byproduct, we can obtain the boundedness of the third term of \eqref{eq.e.w} for $j=0$.
Using it, in the second step, we give estimates of
$E^{(k,j)}[w](t;t_0)$ for $j=1,\ldots,2k$ in order,
and similarly, we have the boundedness of the third term of \eqref{eq.e.w} for $j=1,\ldots,2k$
as outgrowths.
Finally, by using the $2k$-th one, in the step 3, we give an estimate of $E_1^{(k,2k+1)}[w](t;t_0)$.

\begin{proof}[Proof of Lemma \ref{lem.e.w}]
We divide the proof into the following three steps.
\begin{description}
\item[\quad Step 1.] An estimate for $E^{(k,0)}[w](t;t_0,\nu_{k,0})$;
\item[\quad Step 2.] Estimates of $E^{(k,j)}[w](t;t_0,\nu_{k,j})$ for $j=1,\ldots,2k$;
\item[\quad Step 3.] An estimate for $E_1^{(k,2k+1)}[w](t;t_0)$.
\end{description}

\subsection*{Step 1: An estimate for $E^{(k,0)}[w](t;t_0,\nu_{k,0})$}
\begin{lemma}\label{lem.ek.j0.1}
Under the assumption on Theorem \ref{thm.ek}, there exists a constant
$t_{0,1} \ge 1$
such that
for any
$t_0 \ge t_{0,1}$ and $t \ge 0$, we have
\begin{align*}%
	\frac{d}{dt} E_1^{(k,0)} [w](t;t_0)
		\le - \int_{\Omega} a(x) |\partial_t w(x,t)|^2
			\Psi(x,t;t_0)^{\lambda+(2k+1)\frac{\alpha}{2-\alpha}}\,dx
			+ C \int_{\Omega} |\nabla w(x,t)|^2
			\Psi(x,t;t_0)^{\lambda+(2k+1)\frac{\alpha}{2-\alpha}-1}\,dx
\end{align*}%
with some constant
$C=C(k,N,\alpha,\varepsilon,\lambda,t_0) >0$.
\end{lemma}
The proof is completely the same as that of Lemma \ref{lem.e1} and we omit the detail.

\begin{lemma}\label{lem.ek.j0.0}
Under the assumption on Theorem \ref{thm.ek}, there exists a constant $t_{0,2} \ge 1$ such that
for any
$t_0 \ge t_{0,2}$ and $t \ge 0$,
we have
\begin{align*}
	\frac{d}{dt} E_{0}^{(k,0)} [w](t;t_0)
		&\le
		-  \int_{\Omega} |\nabla w(x,t) |^2
			\Psi(x,t;t_0)^{\lambda+2k\frac{\alpha}{2-\alpha}}\,dx
		+ C \int_{\Omega} |\partial_t w(x,t)|^2
			\Psi(x,t;t_0)^{\lambda+2k\frac{\alpha}{2-\alpha}} \,dx \\
		&\quad + C \int_{\Omega} a(x) |w(x,t)|^2
			\Psi(x,t;t_0)^{\lambda + (2k-1) \frac{\alpha}{2-\alpha}} \,dx
\end{align*}
with some constant
$C = C(k,N,\alpha,\varepsilon,\lambda,t_0) > 0$.
\end{lemma}
\begin{proof}
By integration by parts,
we have
\begin{align*}
	\frac{d}{dt} E_{0}^{(k,0)} [w](t;t_0)
		&= 2 \int_{\Omega} (|\partial_t w|^2 + w \Delta w )
			\Psi^{\lambda + 2k\frac{\alpha}{2-\alpha}} \,dx \\
		&\quad + \left( \lambda + 2k\frac{\alpha}{2-\alpha} \right)
			\int_{\Omega} (2 w \partial_t w + a(x) |w|^2)
			\Psi^{\lambda + 2k\frac{\alpha}{2-\alpha}-1} \,dx \\
		&= 2 \int_{\Omega} |\partial_t w|^2
			\Psi^{\lambda + 2k\frac{\alpha}{2-\alpha}} \,dx
			-2 \int_{\Omega} |\nabla w|^2
			\Psi^{\lambda + 2k\frac{\alpha}{2-\alpha}} \,dx \\
		&\quad - 2 \left( \lambda +2k\frac{\alpha}{2-\alpha} \right)
			\int_{\Omega} w (\nabla w \cdot \nabla \Psi)
				\Psi^{\lambda + 2k\frac{\alpha}{2-\alpha}-1} \,dx \\
		&\quad + \left( \lambda + 2k\frac{\alpha}{2-\alpha} \right)
			\int_{\Omega} (2 w \partial_t w + a(x) |w|^2)
			\Psi^{\lambda + 2k\frac{\alpha}{2-\alpha}-1} \,dx \\
		&\le C \int_{\Omega} |\partial_t w|^2
			\Psi^{\lambda + 2k\frac{\alpha}{2-\alpha}} \,dx
			-  \int_{\Omega} |\nabla w|^2
			\Psi^{\lambda + 2k\frac{\alpha}{2-\alpha}} \,dx \\
		&\quad + C \int_{\Omega} a(x) |w|^2
			\Psi^{\lambda + (2k-1)\frac{\alpha}{2-\alpha}} \,dx
\end{align*}
for sufficiently large
$t_0$.
Here we have used the Schwarz inequality
\begin{align*}
	\left| w (\nabla w \cdot \nabla \Psi) \right|
	&\le \eta |\nabla w|^2 \frac{|\nabla \Psi|^2}{a(x)\Psi}
		+ C a(x) |w|^2 \Psi
\end{align*}
with small $\eta>0$
and \eqref{psi.ratio}.
This gives the desired estimate.
\end{proof}

\begin{lemma}\label{lem.ek.j0}
Under the assumptions on Lemma \ref{lem.e.w},
there exist constants
$t_{0,*} \ge \max\{ t_{0,1}, t_{0,2} \}$
and
$\nu_{k,0} = \nu_{k,0} (k,N,\alpha,\varepsilon,\lambda,t_0) > 0$
such that
for any
$t_0 \ge t_{0,*}$,
\begin{align*}
	E^{(k,0)} [w] (t;t_0) &\ge
		c \int_{\Omega}
			(|\nabla w(x,t)|^2 + |\partial_t w(x,t)|^2)
				\Psi(x,t;t_0)^{\lambda+ (2k+1)\frac{\alpha}{2-\alpha}} \,dx \\
		&\quad + c \int_{\Omega} a(x) | w(x,t)|^2
				\Psi(x,t;t_0)^{\lambda+2k\frac{\alpha}{2-\alpha}} \,dx
\end{align*}
and
\begin{align*}
	&E^{(k,0)} [w] (t;t_0)
		+ \int_0^t \int_{\Omega}
			|\nabla w(x,\tau)|^2
				\Psi(x,\tau;t_0)^{\lambda+2k\frac{\alpha}{2-\alpha}} \,dxd\tau \\
	&\quad + \int_0^t \int_{\Omega}
			a(x) |\partial_t w(x,\tau)|^2
				\Psi(x,\tau;t_0)^{\lambda+(2k+1)\frac{\alpha}{2-\alpha}}\,dxd\tau \\
	&\le C (I+K)
\end{align*}
hold for
$t\ge 0$
with some constants
$c = c(k,N,\alpha,\varepsilon,\lambda,t_0,\nu_{k,0}) > 0$
and
$C = C(k,N,\alpha,\varepsilon,\lambda,t_0,\nu_{k,0}) > 0$.
\end{lemma}
\begin{proof}
The first assertion is obvious by taking
$\nu_{k,0}$
sufficiently small.
For the second assertion,
retaking $\nu_{k,0}$ smaller if needed,
using Lemmas \ref{lem.ek.j0.1} and \ref{lem.ek.j0.0}
and taking $t_{0,*} \ge \max\{ t_{0,1}, t_{0,2} \}$
sufficiently large,
we have
\begin{align*}
	&E^{(k,0)} [w](t;t_0)
		+ \int_0^t \int_{\Omega}
			|\nabla w(x,\tau)|^2
				\Psi(x,\tau;t_0)^{\lambda+2k\frac{\alpha}{2-\alpha}} \,dxd\tau \\
	&\quad + \int_0^t \int_{\Omega}
			a(x) |\partial_t w(x,\tau)|^2
				\Psi(x,\tau;t_0)^{\lambda+(2k+1)\frac{\alpha}{2-\alpha}}\,dxd\tau \\
	&\le C E^{(k,0)} [w] (0;t_0)
		+  C \int_0^t \int_{\Omega} a(x) |w(x,\tau)|^2
			\Psi(x,\tau;t_0)^{\lambda + (2k-1) \frac{\alpha}{2-\alpha}} \,dxd\tau.
\end{align*}
The last term is bounded by
$C(I+K)$
thanks to the assumption of Lemma \ref{lem.e.w}.
This leads to the conclusion.
\end{proof}

\subsection*{Step 2: Estimates of $E^{(k,j)}[w](t;t_0,\nu_{k,j})$ for $j=1,\ldots,2k$}

Next, we estimate $E^{(k,j)}[w](t;t_0,\nu_{k,j})$ for $j=1,\ldots,2k$ in order.
The key point is
to apply the boundedness of
\begin{align*}
	\int_0^t (t_0+\tau)^{j-1} \int_{\Omega} a(x) 
	\Big(|\nabla w(x,\tau)|^2+|\pa_tx(x,\tau)|^2\Big)
			\Psi(x,\tau;t_0)^{\lambda + (2k+1-j) \frac{\alpha}{2-\alpha}} \,dx \,d\tau
\end{align*}
obtained in the $(j-1)$-th step to the estimate in the $j$-th step
(see the proof of Lemma \ref{lem.ek.j}).

\begin{lemma}\label{lem.ek.j.1}
Under the assumption on Lemma \ref{lem.e.w}, for
$j=1,\ldots,2k$,
there exists a constant
$t_{j,1} \ge 1$
such that
for any
$t_0 \ge t_{j,1}$ and $t \ge 0$, we have
\begin{align*}%
	\frac{d}{dt} E_1^{(k,j)} [w] (t;t_0)
		&\le - (t_0+t)^j \int_{\Omega} a(x) |\partial_t w(x,t)|^2
			\Psi(x,t;t_0)^{\lambda+(2k+1-j)\frac{\alpha}{2-\alpha}}\,dx \\
		&\quad + C (t_0+t)^{j-1} \int_{\Omega} (|\nabla w(x,t)|^2 + |\partial_t w(x,t)|^2 )
			\Psi(x,t;t_0)^{\lambda+(2k+1-j)\frac{\alpha}{2-\alpha}}\,dx
\end{align*}%
with some constant
$C=C(k,j,N,\alpha,\varepsilon,\lambda,t_0) >0$.
\end{lemma}
\begin{proof}
By integration by parts,
we have
\begin{align*}
	\frac{d}{dt} E_1^{(k,j)} [w] (t;t_0)
	&= 2(t_0+t)^j \int_{\Omega}
		(\nabla \partial_t w \cdot \nabla w + \pa_tw \partial_t^2 w )
		\Psi^{\lambda + (2k+1-j)\frac{\alpha}{2-\alpha}} \,dx \\
	&\quad + (t_0+t)^j \int_{\Omega} (|\nabla w|^2 + |\partial_t w|^2)
		\left[ j(t_0+t)^{-1} + \left(\lambda +(2k+1-j)\frac{\alpha}{2-\alpha} \right) \Psi^{-1} \right]
				\Psi^{\lambda + (2k+1-j)\frac{\alpha}{2-\alpha}} \,dx \\
	&\le
		-2 (t_0+t)^j \int_{\Omega} a(x) |\partial_t w|^2
			\Psi^{\lambda + (2k+1-j)\frac{\alpha}{2-\alpha}} \,dx \\
	&\quad - 2 (t_0+t)^j \left(\lambda +(2k+1-j)\frac{\alpha}{2-\alpha} \right) 
			\int_{\Omega} \partial_t w (\nabla w \cdot \nabla \Psi )
				\Psi^{\lambda + (2k+1-j)\frac{\alpha}{2-\alpha}-1} \,dx \\
	&\quad + C(t_0+t)^{j-1} \int_{\Omega} (|\nabla w|^2 + |\partial_t w|^2)
			\Psi^{\lambda + (2k+1-j)\frac{\alpha}{2-\alpha}} \,dx \\
	&\le - (t_0+t)^j \int_{\Omega} a(x) |\partial_t w|^2
			\Psi^{\lambda + (2k+1-j)\frac{\alpha}{2-\alpha}} \,dx \\
	&\quad + C(t_0+t)^{j-1} \int_{\Omega} (|\nabla w|^2 + |\partial_t w|^2)
			\Psi^{\lambda + (2k+1-j)\frac{\alpha}{2-\alpha}} \,dx
\end{align*} 
for sufficiently large
$t_0$.
Here we have used the Schwarz inequality
\begin{align*}
	\left| \partial_t w (\nabla w \cdot \nabla \Psi ) \right|
	&\le \eta a(x) |\partial_t w|^2 \Psi
		+ C |\nabla w|^2 \frac{|\nabla \Psi|^2}{a(x) \Psi}
\end{align*} 
with small $\eta$ and \eqref{psi.ratio}.
This completes the proof.
\end{proof}

\begin{lemma}\label{lem.ek.j.0}
Under the assumption on Lemma \ref{lem.e.w}, for
$j=1,\ldots,2k$,
there exists a constant $t_{j,2} \ge 1$ such that
for any
$t_0 \ge t_{j,2}$ and $t \ge 0$,
we have
\begin{align*}
	\frac{d}{dt} E_{0}^{(k,j)} [w](t;t_0)
		&\le
		-  (t_0+t)^j \int_{\Omega} |\nabla w(x,t) |^2
			\Psi(x,t;t_0)^{\lambda+(2k-j)\frac{\alpha}{2-\alpha}}\,dx \\
		&\quad + C (t_0+t)^j \int_{\Omega} |\partial_t w(x,t)|^2
			\Psi(x,t;t_0)^{\lambda+(2k-j)\frac{\alpha}{2-\alpha}} \,dx \\
		&\quad + C (t_0+t)^{j-1} \int_{\Omega} a(x) |w(x,t)|^2
			\Psi(x,t;t_0)^{\lambda + (2k-j) \frac{\alpha}{2-\alpha}} \,dx \\
		&\quad + C (t_0+t)^{j-1} \int_{\Omega} |\partial_t w|^2
			\Psi(x,t;t_0)^{\lambda + (2k+1-j) \frac{\alpha}{2-\alpha}} \,dx
\end{align*}
with some constant
$C = C(k,j,N,\alpha,\varepsilon,\lambda,t_0) > 0$.
\end{lemma}
\begin{proof}
By integration by parts, we have
\begin{align*}%
	\frac{d}{dt} E_{0}^{(k,j)} [w] (t;t_0)
		&=
		2 (t_0+t)^j \int_{\Omega} (|\partial_t w|^2 + w \Delta w)
			\Psi^{\lambda +(2k-j)\frac{\alpha}{2-\alpha}} \,dx \\
		&\quad
			+ (t_0+t)^{j} \int_{\Omega} ( 2 w \partial_t w + a(x) |w|^2 )
				\left[ j(t_0+t)^{-1} + \left( \lambda + (2k-j)\frac{\alpha}{2-\alpha} \right) \Psi^{-1} \right]
				\Psi^{\lambda +(2k-j)\frac{\alpha}{2-\alpha}} \,dx \\
		&\le 2 (t_0+t)^j \int_{\Omega} |\partial_t w|^2
				\Psi^{\lambda +(2k-j)\frac{\alpha}{2-\alpha}} \,dx \\
		&\quad
			- 2 (t_0+t)^j \int_{\Omega} |\nabla w|^2
				\Psi^{\lambda +(2k-j)\frac{\alpha}{2-\alpha}} \,dx \\
		&\quad
			-2 (t_0+t)^j \left( \lambda + (2k-j)\frac{\alpha}{2-\alpha} \right)
				\int_{\Omega} w (\nabla w \cdot \nabla \Psi)
					\Psi^{\lambda +(2k-j)\frac{\alpha}{2-\alpha}-1} \,dx \\
		&\quad
			+C(t_0+t)^{j-1} \int_{\Omega} (2|w| |\partial_t w| + a(x) |w|^2 )
				\Psi^{\lambda +(2k-j)\frac{\alpha}{2-\alpha}} \,dx \\
		&\le
			C (t_0+t)^j \int_{\Omega} |\partial_t w|^2
				\Psi^{\lambda +(2k-j)\frac{\alpha}{2-\alpha}} \,dx \\
		&\quad
			- (t_0+t)^j \int_{\Omega} |\nabla w|^2
				\Psi^{\lambda +(2k-j)\frac{\alpha}{2-\alpha}} \,dx \\
		&\quad
			+C (t_0+t)^{j-1}
				\int_{\Omega} a(x) |w|^2
					\Psi^{\lambda +(2k-j)\frac{\alpha}{2-\alpha}} \,dx \\
		&\quad
			+C(t_0+t)^{j-1} \int_{\Omega} |\partial_t w|^2
				\Psi^{\lambda +(2k+1-j)\frac{\alpha}{2-\alpha}} \,dx.
\end{align*}%
Here, we have used the following inequalities
\begin{align*}%
	\left| w (\nabla w \cdot \nabla \Psi) \right|
		&\le \eta |\nabla w|^2 \frac{|\nabla \Psi|^2}{a(x) \Psi}
			+ C a(x) |w|^2 \Psi, 
			\quad
	|w|\,|\pa_t w|\leq a(x)|w|^2+C|\pa_tw|^2\Psi^{\frac{\alpha}{2-\alpha}}
\end{align*}%
and \eqref{psi.ratio}.
This completes the proof.
\end{proof}

\begin{lemma}\label{lem.ek.j}
Under the assumptions on Lemma \ref{lem.e.w},
for each
$j=1,\ldots,2k$,
there exist constants
$t_{j,*} \ge \max\{ t_{j,1}, t_{j,2} \}$
and
$\nu_{k,j} = \nu_{k,j} (k,j,N,\alpha,\varepsilon,\lambda,t_0) > 0$
such that for any $t_0 \ge t_{j,*}$, we have
\begin{align*}
	E^{(k,j)}[w] (t;t_0, \nu_{k,j}) &\ge
		c (t_0+t)^j \int_{\Omega}
			(|\nabla w(x,t)|^2 + |\partial_t w(x,t)|^2)
				\Psi(x,t;t_0)^{\lambda+ (2k+1-j)\frac{\alpha}{2-\alpha}} \,dx \\
		&\quad + c (t_0+t)^j \int_{\Omega} a(x) | w(x,t)|^2
				\Psi(x,t;t_0)^{\lambda+(2k-j)\frac{\alpha}{2-\alpha}} \,dx
\end{align*}
and
\begin{align*}
	&E^{(k,j)} [w] (t;t_0, \nu_{k,j})
		+ \int_0^t (t_0+\tau)^j  \int_{\Omega}
			|\nabla w(x,\tau)|^2
				\Psi(x,\tau;t_0)^{\lambda+(2k-j)\frac{\alpha}{2-\alpha}} \,dxd\tau \\
	&\quad + \int_0^t (t_0+\tau)^j \int_{\Omega}
			a(x) |\partial_t w(x,\tau)|^2
				\Psi(x,\tau;t_0)^{\lambda+(2k+1-j)\frac{\alpha}{2-\alpha}}\,dxd\tau \\
	&\le C (I+K)
\end{align*}
for
$t\ge 0$
with some constants
$c = c(k,j,N,\alpha,\varepsilon,\lambda,t_0,\nu_{k,j}) > 0$
and
$C = C(k,j,N,\alpha,\varepsilon,\lambda,t_0,\nu_{k,0},\ldots,\nu_{k,j}) > 0$.
\end{lemma}
\begin{proof}
The first assertion is obvious by taking $\nu_{k,j}$ sufficiently small.
We prove the second assertion for $j=1,\ldots,2k$ in order.
By Lemmas \ref{lem.ek.j.1} and \ref{lem.ek.j.0},
noting that 
\[
|\pa_tw|^2\Psi(x,\tau;t_0)^{\lambda+(2k+1-j)\frac{\alpha}{2-\alpha}}
\leq 
Ca|\pa_tw|^2\Psi(x,\tau;t_0)^{\lambda+(2k+1-(j-1))\frac{\alpha}{2-\alpha}}
\]
(and retaking $\nu_{k,j}$ if needed), we have
\begin{align*}
	&E^{(k,j)} [w] (t;t_0, \nu_{k,j})
		+ \int_0^t (t_0+\tau)^j \int_{\Omega}
			|\nabla w(x,\tau)|^2
				\Psi(x,\tau;t_0)^{\lambda+(2k-j)\frac{\alpha}{2-\alpha}} \,dxd\tau \\
	&\quad + \int_0^t (t_0+\tau)^j \int_{\Omega}
			a(x) |\partial_t w(x,\tau)|^2
				\Psi(x,\tau;t_0)^{\lambda+(2k+1-j)\frac{\alpha}{2-\alpha}}\,dxd\tau \\
	&\le C E^{(k,j)}[w] (0;t_0, \nu_{k,j})
		+  C \int_0^t (t_0+\tau)^{j-1} \int_{\Omega} a(x) | w(x,\tau)|^2
			\Psi(x,\tau;t_0)^{\lambda + (2k-j) \frac{\alpha}{2-\alpha}} \,dx \\
	&\quad + C \int_0^t (t_0+\tau)^{j-1} \int_{\Omega} (|\nabla w|^2 + |\partial_t w|^2)
				\Psi(x,\tau;t_0)^{\lambda+(2k+1-j)\frac{\alpha}{2-\alpha}}\,dxd\tau
\end{align*}
We easily see that
$E^{(k,j)}[w] (0;t_0, \nu_{k,j}) \le CI$
with some
$C=C(k,j,N,\alpha,\varepsilon,\lambda,t_0,\nu_{k,j}) > 0$.
The second term in the right-hand side is bounded by
$C K$
thanks to the assumption \eqref{induction.assumption}.
Moreover, the third term in the right-hand side is bounded by
$C (I+K)$
because of the assertion for the case $j-1$
(when $j=1$ we apply Lemma \ref{lem.ek.j0}).
Continuing this argument from $j=1$ to $j=2k$, we reach the conclusion.
\end{proof}

\subsection*{Step 3: An estimate for $E_1^{(k,2k+1)}[w](t;t_0)$}

Finally, we show the boundedness of $E_1^{(k,2k+1)}[w](t;t_0)$,
which gives the desired decay for $w$.
\begin{lemma}\label{lem.ek.2k+1.1}
Under the assumption on Lemma \ref{lem.e.w}, there exists a constant
$t_{2k+1,1} \ge 1$
such that
for any
$t_0 \ge t_{2k+1,1}$ and $t \ge 0$, we have
\begin{align*}%
	E_1^{(k,2k+1)} [w] (t;t_0)
		+ \int_0^t (t_0+\tau)^{2k+1} \int_{\Omega} a(x) |\partial_t w(x,\tau)|^2 \Psi(x,\tau;t_0)^{\lambda}\,dx
		&\le
		C(I+K)
\end{align*}%
with some constant
$C=C(k,N,\alpha,\varepsilon,\lambda,t_0,\nu_{k,0},\ldots,\nu_{k,2k}) >0$.
\end{lemma}
\begin{proof}
We first have
\begin{align*}%
	\frac{d}{dt} E_1^{(k,2k+1)} [w](t;t_0)
		&\le - (t_0+t)^{2k+1} \int_{\Omega} a(x) |\partial_t w(x,t)|^2
			\Psi(x,t;t_0)^{\lambda}\,dx \\
		&\quad + C (t_0+t)^{2k} \int_{\Omega} (|\nabla w(x,t)|^2 + |\partial_t w(x,t)|^2 )
			\Psi(x,t;t_0)^{\lambda}\,dx,
\end{align*}%
which is proved by the same way as Lemma \ref{lem.ek.j.1} and we omit the detail.
Integrating the above on $[0,t]$, we have
\begin{align*}%
	&E_1^{(k,2k+1)} [w] (t;t_0)
		+ \int_0^t (t_0+\tau)^{2k+1} \int_{\Omega} a(x) |\partial_t w|^2
			\Psi^{\lambda}\,dxd\tau \\
	&\le E_1^{(k,2k+1)} [w] (0;t_0)
		+ C \int_0^t (t_0+\tau)^{2k} \int_{\Omega} (|\nabla w|^2 + |\partial_t w|^2 )
			\Psi^{\lambda}\,dxd\tau.
\end{align*}%
The right-hand side is bounded by
$C (I+K)$
thanks to Lemma \ref{lem.ek.j} and 
the inequality
$|\pa_tw|^2\Psi^\lambda\leq Ca(x)|\pa_tw|^2\Psi^{\lambda+\frac{\alpha}{2-\alpha}}$.
The proof is complete.
\end{proof}

Finally, combining Lemmas \ref{lem.ek.j0}, \ref{lem.ek.j}, \ref{lem.ek.2k+1.1},
we have the assertion of Lemma \ref{lem.e.w}.
\end{proof}

\subsection{Diffusion phenomena}

To close this paper, we finally consider the 
asymptotic profile of solutions to \eqref{dw}. 
From the viewpoint of weighted energy estimates proved in the previous subsection, 
we expect that the solution of \eqref{dw} behaves like 
the one of \eqref{de} at $t\to \infty$. 

The following is the statement for diffusion phenomena for the problem \eqref{dw}.
The assertion for the case $a(x)\not \equiv |x|^{-\alpha}$ is 
an improvement of \cite[Theorem 1.2]{SoWa19_CCM} in which the spatial case $a(x)=|x|^{-\alpha}$ is studied.
\begin{theorem}\label{thm.5.20}
Let
$(u_0, u_1) \in (H^3(\Omega)\cap H^1_0(\Omega)) \times (H^2(\Omega) \cap H^1_0(\Omega))$
satisfies the compatibility condition of order $1$.
Let
$\delta \in (0,1/2), \varepsilon \in(0,1), \lambda \in (0, (1-2\delta)\gamma_{\varepsilon})$
with $\gamma_{\varepsilon}$ defined by \eqref{gammatilde}.
Assume
\begin{align*}
	I_{1} = \sum_{\ell=0}^1
		\left[ \int_{\Omega} (|\nabla u_{\ell}(x)|^2 +| u_{\ell+1} (x)|^2 )
		\Psi(x,0;t_0)^{\lambda + (2\ell+1) \frac{\alpha}{2-\alpha}} \,dx
		+ \int_{\Omega} a(x) |u_{\ell}(x)|^2
			\Psi(x,0;t_0)^{\lambda + 2\ell \frac{\alpha}{2-\alpha}} \,dx \right] < \infty.
\end{align*}
Then, we have the asymptotic estimate
\begin{align*}
	\| u(t) - T(t)[u_0+a^{-1}u_1] \|_{L^2_{d\mu}} \le C (1+t)^{-\lambda/2} \eta(t) \sqrt{I_1},
	\quad t\geq 1, 
\end{align*}
where
\begin{align*}
	\eta (t) = \begin{cases}
		(1+t)^{-\frac{2(1-\alpha)}{2-\alpha}} \sqrt{\log (2+t)}
			&\text{if}\ \lambda \in [\frac{2\alpha}{2-\alpha}, \frac{N-\alpha}{2-\alpha}),
			\\[3pt]
		(1+t)^{-\frac{2(1-\alpha)}{2\alpha} \lambda}
			&\text{if}\ \lambda \in (0, \frac{2\alpha}{2-\alpha}).
		\end{cases}
\end{align*}
\end{theorem}

\begin{proof}
First, by the same argument as \cite[Lemma 5.1]{SoWa19_CCM}, we can show that
$u(t)$ belongs to $D(L)$ defined in Lemma \ref{lem.4.1} and
$a(x)^{-1} \partial_t^2 u \in L^{\infty}(0,\infty; L^2_{d\mu})$.
Thus, rewriting the equation \eqref{dw} as
\begin{align*}
	\partial_t u - a(x)^{-1} \Delta u = - a(x)^{-1} \partial_t^2 u,
\end{align*}
and using the semigroup $T(t)$ defined in Lemma \ref{lem.4.1},
we have the integral formula
\begin{align*}
	u(t) = T(t)u_0 - \int_0^t T(t-s) [ a^{-1} \partial_t^2 u(s) ]\,ds
\end{align*}
(see \cite[Lemma 4.1]{SoWa16_JDE} for the detail).
Moreover, by integration by parts, we deduce
\begin{align*} 
	u(t)
	&=T(t)[u_0+a^{-1}u_1]
		- \int_{\frac{t}{2}}^tT(t-s)[a^{-1}\pa_t^2u(s)]\,ds \\
	&\quad - T(t/2)[a^{-1}\pa_tu(t/2)]
		- \int_{0}^{\frac{t}{2}}LT(t-s)[a^{-1}\pa_tu(s)]\,ds
\end{align*}
(see \cite[p.5715]{SoWa16_JDE} for the detail).
Therefore, we obtain the representation
\begin{align*}
	u(t) - T(t)[u_0+a^{-1}u_1]
		= J_1(t) + J_2(t) + J_3(t),
\end{align*}
where
\begin{align*}
	J_1(t) &= - \int_{\frac{t}{2}}^tT(t-s)[a^{-1}\pa_t^2u(s)]\,ds,\\
	J_2(t) &= - T(t/2)[a^{-1}\pa_tu(t/2)],\\
	J_3(t) &= - \int_{0}^{\frac{t}{2}}LT(t-s)[a^{-1}\pa_tu(s)]\,ds.
\end{align*}
Hence, it suffices to estimate $J_1, J_2$ and $J_3$ term by term.
In what follows, we shall frequently apply Theorem \ref{thm.ek} with $k=1$.
In the rest of the proof, we divide the proof into two cases
$\lambda \in [\frac{2\alpha}{2-\alpha}, \frac{N-\alpha}{2-\alpha})$
and
$\lambda \in (0, \frac{2\alpha}{2-\alpha})$.

{\bf (The case $\lambda \in [\frac{2\alpha}{2-\alpha}, \frac{N-\alpha}{2-\alpha})$)}
By the Schwarz inequality, Lemma \ref{lem.4.1} and
Theorem \ref{thm.ek} with the bound of
$\int_0^t (t_0+s)^3 \int_{\Omega} a(x) |\partial_t^2 u(x,s)|^2 \Psi(x,s;t_0)^{\lambda} \,dxds$,
we deduce
\begin{align*}
	\| J_1(t) \|_{L^2_{d\mu}}^2
	&\le \frac{t}{2} \int^t_{t/2} \left\| a^{-1} [ \sqrt{a} \partial_t^2 u(s) ] \right\|_{L^2}^2 \,ds \\
	&\le C t \int^t_{t/2} \left\| \Psi^{\frac{\alpha}{2-\alpha}} [ \sqrt{a} \partial_t^2 u(s) ] \right\|_{L^2}^2 \,ds \\
	&\le C t \int^t_{t/2} (t_0+s)^{-3 - (\lambda-\frac{2\alpha}{2-\alpha})}
		\left[ (t_0+s)^3 \int_{\Omega} a(x) |\partial_t^2 u(x,s)|^2 \Psi(x,s;t_0)^{\lambda} \,dx \right] \,ds \\
	&\le C (t_0+t)^{-\lambda - \frac{4(1-\alpha)}{2-\alpha}} I_1,
\end{align*}
and hence,
\begin{align*}
	\| J_1(t) \|_{L^2_{d\mu}}
		&\le C (t_0+t)^{-\frac{\lambda}{2} - \frac{2(1-\alpha)}{2-\alpha}} \sqrt{I_1}.
\end{align*}
Applying Lemma \ref{lem.4.1} and Theorem \ref{thm.ek} with
the bound of
$(t_0+t)^2 \int_{\Omega} a(x) |\partial_t u(x,t)|^2 \Psi(x,t;t_0)^{\lambda} \,dx$,
we have
\begin{align*}
\|J_2(t)\|_{L^2_{d\mu}}
&\leq \| a^{-1} [ \pa_tu(t/2)] \|_{L^2}
\\
&\leq 
C \| \Psi^{\frac{\alpha}{2-\alpha}} [\sqrt{a} \pa_tu(t/2)] \|_{L^2}
\\
&\leq 
	C(t_0+t/2)^{-\frac{1}{2}( \lambda - \frac{2\alpha}{2-\alpha}) - 1}
		\left[ (t_0+t/2)^2 \int_{\Omega} a(x) |\partial_t u(x,t/2)|^2 \Psi(x,t/2;t_0)^{\lambda}\,dx \right]^{\frac{1}{2}}
\\
&\leq 
	C (t_0+t)^{-\frac{\lambda}{2} - \frac{2(1-\alpha)}{2-\alpha}} \sqrt{I_1}.
\end{align*}
By Proposition \ref{prop.4.2} with $\sigma=\frac{\lambda}{2(2-\alpha)}-\alpha$, the Schwarz inequality
and Theorem \ref{thm.e1}, we also estimate
\begin{align*}
&\|J_3(t) \|_{L^2_{d\mu}}
\\
&\leq
	\int_0^{t/2} \left\| LT\left(\frac{t-s}{2}\right)T\left(\frac{t-2s}{4}\right)T\left(\frac{t}{4}\right)
		[a^{-1}\pa_tu(s)]\right\|_{L^2_{d\mu}}\,ds
\\
&\leq 
\int_0^{t/2}(t-s)^{-1}\left\|T\left(\frac{t}{4}\right)[a^{-1}\pa_tu(s)]\right\|_{L^2_{d\mu}}\,ds 
\\
&\leq 
C\int_0^{t/2}(t-s)^{-1}\left(t_0+\frac{t}{4}\right)^{-(\frac{\lambda}{2} - \frac{\alpha}{2-\alpha})}
\left\|\lr{x}^{\frac{(2-\alpha)\lambda}{2}}[\sqrt{a}\pa_tu(s)]\right\|_{L^2}\,ds 
\\
&\leq 
C(t_0+t)^{-(\frac{\lambda}{2} - \frac{\alpha}{2-\alpha})-1}
	\left( \int_0^{t/2} (t_0+s)^{-1}\,ds \right)^{\frac{1}{2}}
	\left[ \int_0^{t/2} (t_0+s) \int_{\Omega} a(x) |\partial_t u(x,s)|^2 \Psi(x,s;t_0)^{\lambda} \,dx ds \right]^{\frac{1}{2}}
\\
&\leq 
C(t_0+t)^{-\frac{\lambda}{2} - \frac{2(1-\alpha)}{(2-\alpha)}} \sqrt{ \log(t_0+t) } \sqrt{I_1}.
\end{align*}

{\bf (The case $\lambda \in (0, \frac{2\alpha}{2-\alpha})$)}
In this case we shall use the interpolation estimates
\begin{align}
\label{inter1}
	(t_0+t)^{2(1-\theta)} \int_{\Omega} a(x)|\partial_t u(x,t)|^2 \Psi(x,t;t_0)^{\lambda + \frac{2\alpha}{2-\alpha} \theta } \,dx
	\le C  I_1,\\
\label{inter2}
	\int_0^t (t_0+s)^{3(1-\theta)} \int_{\Omega} a(x) |\partial_t^2 u(x,s)|^2
		\Psi(x,s;t_0)^{\lambda + \frac{3\alpha}{2-\alpha} \theta} \,dx ds \le C I_1
\end{align}
for
$\theta \in [0,1]$,
which follow from Theorem \ref{thm.ek} with $k=1$ and the H\"{o}lder inequality.

For $J_1(t)$, applying \eqref{inter2} with
$\theta = \frac{2}{3} - \frac{2-\alpha}{3\alpha} \lambda$, we compute
\begin{align*}
	\| J_1(t) \|_{L^2_{d\mu}}^2
		&\le \frac{t}{2} \int_{t/2}^t \| a^{-1} [ \sqrt{a} \partial_t^2 u ] \|_{L^2}^2 \,ds \\
		&\le C t \int_{t/2}^t \| \Psi^{\frac{\alpha}{2-\alpha}} [ \sqrt{a} \partial_t^2 u ] \|_{L^2}^2 \,ds \\
		&\le C t \int_{t/2}^t \int_{\Omega} a(x) |\partial_t^2 u(x,s)|^2
				\Psi(x,s;t_0)^{\lambda + (\frac{2\alpha}{2-\alpha} - \lambda)} \,dx ds \\
		&\le C t \int_{t/2}^t 
			(t_0+s)^{-3(1-\theta)}
			\cdot (t_0+s)^{3(1-\theta)} 
			\int_{\Omega} a(x) |\partial_t^2 u(x,s)|^2
			\Psi(x,s;t_0)^{\lambda + \frac{3\alpha}{2-\alpha}\theta } \,dx ds \\
		&\le C t (t_0+t/2)^{-3(1-\theta)} I_1 \\
		&\le C (t_0+t)^{- \frac{2-\alpha}{\alpha} \lambda} I_1,
\end{align*}
and hence,
\begin{align*}
	\| J_1(t) \|_{L^2_{d\mu}} \le C  (t_0+t)^{- \frac{2-\alpha}{2\alpha} \lambda} \sqrt{I_1}.
\end{align*}
For $J_2(t)$, we have
\begin{align*}
\|J_2(t)\|_{L^2_{d\mu}}
&\leq \| a^{-1} [ \sqrt{a} \pa_tu(t/2)]\|_{L^2}
\\
&\leq 
C \| \Psi^{\frac{\alpha}{2-\alpha}} [\sqrt{a} \pa_tu(t/2)] \|_{L^2}
\\
&\leq 
	C \left[ \int_{\Omega} a(x) |\partial_t u(x,t/2)|^2
			\Psi(x,t/2;t_0)^{\lambda + (\frac{2\alpha}{2-\alpha} - \lambda) }\,dx \right]^{\frac{1}{2}}
\\
&\leq 
	C (t_0+t/2)^{-(1-\theta)}
			\left[
				(t_0+t/2)^{2(1-\theta)}
				\int_{\Omega} a(x) |\partial_t u(x,t/2)|^2
			\Psi(x,t/2;t_0)^{\lambda + \frac{2\alpha}{2-\alpha}\theta }\,dx \right]^{\frac{1}{2}}
\\
&\leq
	C(t_0+t)^{-\frac{2-\alpha}{2\alpha}\lambda} \sqrt{I_1},
\end{align*}
where we used \eqref{inter1} with $\theta = 1 - \frac{2-\alpha}{2\alpha} \lambda$.
Finally, for $J_3(t)$, we estimate
\begin{align*}
\|J_3(t) \|_{L^2_{d\mu}}
&\leq
	\int_0^{t/2} \left\| LT\left(\frac{t-s}{2}\right)T\left(\frac{t-2s}{4}\right)T\left(\frac{t}{4}\right)
		[a^{-1}\pa_tu(s)]\right\|_{L^2_{d\mu}}\,ds
\\
&\leq 
\int_0^{t/2}(t-s)^{-1} \left\| a^{-1} \pa_tu(s)\right\|_{L^2_{d\mu}}\,ds
\\
&\leq 
C t^{-1} \int_0^{t/2} \left\| \Psi^{\frac{\alpha}{2-\alpha}} [\sqrt{a} \pa_tu(s) ]\right\|_{L^2}\,ds
\\
&\leq
C t^{-1} \int_0^{t/2} \left[ \int_{\Omega} a(x) |\partial_t u(x,s)|^2
	\Psi(x,s;t_0)^{\lambda + (\frac{2\alpha}{2-\alpha} - \lambda)} \,dx \right]^{\frac{1}{2}} ds
\\
&\leq
C t^{-1} \int_0^{t/2} (t_0+s)^{-(1-\theta)} 
	\left[ (t_0+s)^{2(1-\theta)} \int_{\Omega} a(x) |\partial_t u(x,s)|^2
	\Psi(x,s;t_0)^{\lambda + \frac{2\alpha}{2-\alpha}\theta } \,dx \right]^{\frac{1}{2}} ds
\\
&\leq
C (t_0+t)^{-1 - (1-\theta) + 1} \sqrt{I_1}
\\
&\leq
C (t_0+t)^{-\frac{2-\alpha}{2\alpha} \lambda} \sqrt{I_1},
\end{align*}
where
$\theta = 1 - \frac{2-\alpha}{2\alpha} \lambda$.
Combining all the estimates, we have the desired estimate.
\end{proof}

\begin{proof}[Proof of Theorem \ref{thm:dif-pheno}]
For the initial data
$(u_0, u_1) \in (H^2(\Omega) \cap H^1_0(\Omega)) \times H^1_0(\Omega)$
satisfying
$E_0, E_0' < \infty$
with $\sigma \in (0,\frac{N-\alpha}{2})$,
we take an appropriate approximation
$\{ (u_{0n}, u_{1n}) \}_{n=1}^{\infty}$
satisfying
$\| (u_{0n}, u_{1n}) \|_{H^2 \times H^1} \le C \| (u_0, u_1) \|_{H^2 \times H^1}$
and the assumptions of Theorem \ref{thm.5.20}
with $\lambda = \frac{2\sigma}{2-\alpha}$
(see \cite[p.611]{SoWa18_ADE} for the detail).
Then, applying Theorem \ref{thm.5.20} to $(u_{0n}, u_{1n})$
and taking the limit $n\to \infty$,
we have the desired conclusion.
\end{proof}

\subsection*{Acknowledgements}

This work was supported by 
JSPS KAKENHI Grant Numbers 
JP18K134450, JP16K17625, JP18H01132.

\newpage

\end{document}